\documentclass{siamart1116}

\usepackage{verbatim}
\usepackage{amsmath} 
\usepackage{color} 
\usepackage{url} 
\usepackage{amssymb,bm}
\usepackage{tabto}
\usepackage[section]{placeins}
\usepackage{wrapfig}
\usepackage{stackengine}
\usepackage{array}
\usepackage{listings}
\usepackage{color} 
\definecolor{mygreen}{RGB}{28,172,0} 
\definecolor{mylilas}{RGB}{170,55,241}
\usepackage{esdiff}
\renewcommand{\vec}{\bm}
\newcommand{\mat}{\mathbf}
\usepackage{amsfonts}
\usepackage{graphicx}
\usepackage{epstopdf}
\usepackage{algorithmic}
\usepackage{nomencl}
\usepackage{tabu, multirow}
\usepackage{subcaption}
\usepackage{array}
\newcolumntype{P}[1]{>{\centering\arraybackslash}p{#1}}
\makenomenclature
\ifpdf
  \DeclareGraphicsExtensions{.eps,.pdf,.png,.jpg}
\else
  \DeclareGraphicsExtensions{.eps}
\fi

\numberwithin{theorem}{section}

\newcommand{\TheTitle}{Robust Designs Via Geometric Programming} 
\newcommand{\TheAuthors}{A. Saab, E. Burnell, and W. W. Hoburg}

\headers{\TheTitle}{\TheAuthors}

\title{{\TheTitle}}

\author{
  Ali Saab\thanks{Department of Computations for Design and Optimization, Massachusetts Institute of Technology, Cambridge, MA
    (\email{saab@mit.edu}).}
  \and
  Edward Burnell\thanks{Department of Mechanical Engineering, Massachusetts Institute of Technology, Cambridge, MA (\email{eburn@mit.edu}).}
  \and
  Warren W. Hoburg\thanks{Department of Aeronautics and Astronautics, Massachusetts Institute of Technology, Cambridge, MA (\email{whoburg@mit.edu}).}
}

\usepackage{amsopn}

\ifpdf
\hypersetup{
  pdftitle={\TheTitle},
  pdfauthor={\TheAuthors}
}
\fi

\begin{document}

\maketitle

\begin{abstract} \label{abs}
An approximate formulation of a robust geometric program (RGP) as a convex program is proposed. Interest in using geometric programs (GPs) to model complex engineering systems has been growing, and this has motivated explicitly modeling the uncertainties fundamental to engineering design optimization. RGPs provide a framework for modeling and solving GPs while representing their uncertainties as belonging to an uncertainty set. The RGP methodologies presented here are based on reformulating the GP and then robustifying it with methods from robust linear programming. These new methodologies, along with previous ones from the literature, are used to robustify two aircraft design problems, and the results of these different methodologies are compared and discussed.
\end{abstract}

\begin{keywords}
  Geometric Programming, Posynomials, Robust Optimization, Uncertainty Set, Robust Linear Programming, Safe Approximation, Aircraft, Conic Optimization
\end{keywords}

\begin{AMS}
  90C90, 90C99
\end{AMS}

\section{Introduction} \label{intro}
A \textbf{geometric program in posynomial form} is a log-convex optimization problem of the form:
\begin{equation}
\begin{aligned}
	& \text{minimize} && f_0 \left(\vec{u}\right) \\
	& \text{subject to} && f_i \left(\vec{u}\right) \leq 1, i = 1,...,m_p\\
	& && h_i \left(\vec{u}\right) = 1, i = 1, ...,m_e\\
\end{aligned}
\label{GP_standard}
\end{equation}
where each $f_i$ is a {\em posynomial} and each $h_i$ is a {\em monomial}. A monomial $h(\vec{u})$ is a function of the form:
\begin{displaymath}
	h(\vec{u}) = e^{b}\textstyle{\prod}_{j=1}^{n}{u_j}^{a_j}
\end{displaymath}
where $\vec{a}$ is a row vector in $\mathbb{R}^n$, $\vec{u}$ is a column vector in $\mathbb{R}^n_+$ , and $b$ is in $\mathbb{R}$. A posynomial $f(\vec{u})$ is the sum of $K \in \mathbb{Z}^+$ monomials:
\begin{displaymath}
	f(\vec{u}) = \textstyle{\sum_{k=1}^{K}}e^{b_{kj}}\prod_{j=1}^{n}{u_j}^{a_{kj}}
\end{displaymath}
where $\vec{a_{k}}$ is a row vector in $\mathbb{R}^n$, $\vec{u}$ is a column vector in $\mathbb{R}^n_+$, and $b_{k}$ is in $\mathbb{R}$ \cite{GP_tutorial}.\\
A logarithmic change of the variables $x_j = \log(u_j)$ would turn a monomial into an {\em exponential of an affine function} and a posynomial into a {\em sum of exponentials of affine functions}. A transformed monomial $h_i(\vec{x})$ is a function of the form:
\begin{displaymath}
    h_i(\vec{x}) = e^{\vec{a_i}\vec{x} + b_i}
\end{displaymath}
where $\vec{a_i}$ is a row vector in $\mathbb{R}^n$, $\vec{x}$ is a column vector in $\mathbb{R}^n$ , and $b_i$ is in $\mathbb{R}$. A transformed posynomial $f_i(\vec{x})$ is the sum of $K_i \in \mathbb{Z}^+$ monomials:
\begin{displaymath}
    f_i(\vec{x}) = \textstyle{\sum_{k=1}^{K_i}}e^{\vec{a_{ik}}\vec{x} + b_{ik}}
\end{displaymath}
where $\vec{a_{ik}}$ is a row vector in $\mathbb{R}^n$, $\vec{x}$ is a column vector in $\mathbb{R}^n$, and $b_{ik}$ is in $\mathbb{R}$.

Solving large-scale GPs became extremely efficient and more reliable \cite{GP_tutorial} after the development of interior point methods \cite{nesterov_nemirovskii_1994} \cite{kortanek_xu_ye_1997}, prompting engineers to use GP formulations for the modeling and design optimization of complex engineering systems. Recent applications of GP have been found in fields as varied as digital circuit optimization \cite{boyd_kim_patil_horowitz_2005}, analog circuit design \cite{hershenson_2004}, communication systems \cite{chiang_2005}, chemical process control \cite{wall_greening_woolsey_1986}, environmental quality control \cite{greenberg_1995}, statistics \cite{mazumdar_jefferson_1983}, as well as antenna and aircraft design \cite{babakhani_lavaei_doyle_hajimiri_2010} \cite{hoburg_abbeel_2014}.

While these GPs can be solved to global optimality for fixed parameter values, many of those values are uncertain during the engineering design process, either because of intrinsic uncertainty (e.g. wind velocity), or unpredictability at that stage of the design process (e.g. the final assembled weight of an airplane's avionics). Designers, seeking a solution they can trust enough to manufacture, account for these uncertainties by making their models conservative. This conservativeness almost always takes the form of specifying worst-case values and margins for all uncertain constraints and parameters \cite{hoburg_abbeel_2014}. Such an approach exaggerates the impact of uncertainties, leading to sub-optimality which is undesired. From a designer's perspective, however, this unwanted conservativeness is a way to ensure robust results. New techniques that can represent design uncertainties less conservatively are thus able to directly improve engineering solutions and present new design opportunities.

Robust optimization has been used to efficiently solve optimization problems under uncertainty \cite{soyster_1973}. In constrast to stochastic optimization \cite{birge_louveaux_2011} \cite{prékopa_2010}, probability density functions are replaced in robust optimization by uncertainty sets. In the late 1990s, Ben-Tal, Nemirovski, and El Ghaoui proved the tractability of robust linear, quadratic, and second order conic programming problems \cite{ben-tal_ghaoui_nemirovski_2009}. Later, Bertsimas et al. derived a tractable approximation of a subset of robust conic optimization problems \cite{bertsimas_sim_2005}, and discussed the importance and means of choosing an uncertainty set to best describe the problem while preserving tractability \cite{bandi_bertsimas_2012}.

Robust geometric programming is co-NP hard in its natural posynomial form \cite{RGPcoNP}, but an interesting tractable formulation for an approximate robust geometric program was presented by Hsiung, Kim, and Boyd \cite{hsiung_kim_boyd_2007}. They proposed a bivariate safe approximation for the posynomial constraints, replacing each posynomial constraint with a set of two-term (``bivariate'') posynomial constraints that can be easily approximated using piecewise-linear functions. Although it simplifies the problem into a linear optimization problem, the piecewise-linear approximation of every posynomial leads to a considerable number of additional constraints. Moreover, because the bivariate safe approximation decouples all but two monomials in each posynomial, it yields an effective uncertainty set larger than the specified one and thus more conservative. Our methods seek to improve on \cite{hsiung_kim_boyd_2007} by reducing the number of constraints while more accurately approximating the uncertainty set to achieve better (less conservative) solutions.

In this article we use three novel methodologies to construct a tractable approximation of a robust geometric program while drastically reducing the number of constraints compared to \cite{hsiung_kim_boyd_2007}. Our first approach deals with each monomial while separating it from its posynomial. Our second approach assumes that only the coefficients $b$ are uncertain and deals with posynomials as linear constraints. Our third approach accounts for dependency between monomials while letting both the coefficients $b$ and the exponents $\vec{a}$ be uncertain.

In Section \ref{RGP} of this article, we define robust geometric programming and derive its robust counterparts. Section \ref{Conservative} introduces a simple formulation, our most conservative proposal, which is then refined in Section \ref{EqIntFor} with the concept of monomial partitioning. Section \ref{twoTerm} briefly reviews the piecewise-linearization of two term posynomials in \cite{hsiung_kim_boyd_2007}, an essential tool for our work. Section \ref{k_term} discusses different methodologies for dealing with larger posynomials. In Section \ref{apps} we apply all methodologies, both new and old, to several previously published GPs for aircraft design and discuss the results. Finally, Section \ref{conc} states our contributions and conclusions.

\section{Robust Geometric Programming} \label{RGP}
The data in a geometric program is usually prone to uncertainties that could either have a probabilistic description or belong to an uncertainty set. Robust geometric programming assumes the latter, and solves the problem for the worst case scenario by finding the best solution that is feasible to all possible realizations from the uncertainty set. This section introduces different forms of GPs, derives the intractable RGP formulations, and then states the tractable approximation presented in \cite{hsiung_kim_boyd_2007}.
\subsection{Other Forms of a Geometric Program}

By replacing each monomial equality constraint $h\left(\vec{x}\right) = 1$ by the two monomial inequality constraints $h\left(\vec{x}\right) \leq 1$ and $\frac{1}{h\left(\vec{x}\right)} \leq 1$, we arrive at the \textbf{inequality form} of a GP:
\begin{equation}
\begin{aligned}
	& \min && \textstyle{\sum}_{k=1}^{K_0}e^{\vec{a_{0k}}\vec{x} + b_{0k}} \\
	& \text{s.t.} && \textstyle{\sum}_{k=1}^{K_i}e^{\vec{a_{ik}}\vec{x} + b_{ik}} \leq 1 \quad \forall i \in 1,...,m\\
\end{aligned}
\label{GP_inequality}
\end{equation}
Without loss of generality, the objective function will be assumed deprived of any data (i.e. replacing it with its epigraph form). Let $\mat{A}$ be the exponents of a GP with components $\vec{a_{ik}}$, and $\vec{b}$ be the coefficients of a GP with components $b_{ik}$.

Also relevant is the \textbf{convex form} of a GP, where a logarithm is applied to the objective function and to the constraints:
\begin{equation}
\begin{aligned}
& \min &&\log\left(f_0\left(\vec{x}\right)\right)\\
& \text{subject to} &&\log\left(\textstyle{\sum}_{k=1}^{K_i}e^{\vec{a_{ik}}\vec{x} + b_{ik}}\right) &&\leq 0 &&\forall i \in 1,...,m
\end{aligned}
\label{GP_convex}
\end{equation}
Note that a GP in its convex form is a convex program whose dual is the entropy maximization problem \cite{woodyThesis}.

A final interesting form of a GP is the \textbf{categorized form}, which will be introduced later in Section \ref{EqIntFor}, where the constraints of a GP are categorized into different sets according to the number of monomial terms in each posynomial constraint.
\subsection{Robust Counterparts}
As mentioned before, a formulation immune to uncertainty in the system's data should be derived. The data ($\mat{A}$ and $\vec{b}$) will be contained in an uncertainty set $\mathcal{U}$, where $\mathcal{U}$ is parameterized affinely by perturbation vector $\vec{\zeta}$ as follows:
\begin{equation}
	\mathcal{U} = \left\{\left[\mat{A};\vec{b}\right] = \left[\mat{A}^0;\vec{b}^0\right] + \textstyle{\sum_{l=1}^{L}\zeta_l\left[\mat{A}^l;\vec{b}^l\right]}\right\}
	\label{Data}
\end{equation}
where $\vec{\zeta}$ belongs to a conic perturbation set $\mathcal{Z} \in \mathbb{R}^L$ paramaterized by $\mat{F},\,\mat{G},\,\vec{h}$ and $\textbf{K}$ such that
\begin{equation}
\mathcal{Z} = \left\{ \vec{\zeta} \in \mathbb{R}^L: \exists \vec{u} \in \mathbb{R}^k:\mat{F}\vec{\zeta} + \mat{G}\vec{u} + \vec{h} \in \textbf{K} \right\}
\label{perturbation_set}
\end{equation}
where $\mathbf{K}$ is a Regular cone in $\mathbb{R}^N$ with a non-empty interior if it is not polyhedral, $\mat{F} \in \mathbb{R}^{N \times L}$, $\mat{G} \in \mathbb{R}^{N \times k}$, and $\vec{h} \in \mathbb{R}^{N}$.

Accordingly, the robust counterpart of the uncertain geometric program in \eqref{GP_inequality} is:
\begin{equation}
\begin{aligned}
& \min &&f_0\left(\vec{x}\right)\\
& \text{subject to} &&\textstyle{\sum}_{k=1}^{K_i}e^{\vec{a_{ik}}\left(\zeta\right)\vec{x} + b_{ik}\left(\zeta\right)} &&\leq 1 &&\forall i \in 1,...,m && \forall \vec{\zeta} \in \mathcal{Z}
\end{aligned}
\label{GP_counterparts}
\end{equation}
These constraints state that the robust optimal solution should be feasible for all possible realizations of the perturbation vector $\vec{\zeta}$. However, the above is a semi-infinite optimization problem, that is, an optimization problem with finite number of variables and infinite number of constraints. Such a problem is intractable using current solvers, and so an equivalent finite set of constraints is usually derived:
\begin{equation}
\begin{aligned}
& \min &&f_0\left(\vec{x}\right)\\
& \text{subject to} &&\max_{\vec{\zeta} \in \mathcal{Z}} \left\{\textstyle{\sum}_{k=1}^{K_i}e^{\vec{a_{ik}}\left(\zeta\right)\vec{x} + b_{ik}\left(\zeta\right)}\right\} &&\leq 1 &&\forall i \in 1,...,m\\
\end{aligned}
\label{GP_counterparts_finite}
\end{equation}
Here each constraint is a maximization problem, and in linear programs can be replaced by its dual. Unfortunately such a replacement will not work in the case of geometric programming. Therefore we will gradually construct a tractable approximation of this robust geometric program throughout this article, with a focus on polyhedral and elliptical uncertainty sets.

\subsection{Two-Term Formulation}
The current state-of-the-art in approximating RGPs is found in \cite{hsiung_kim_boyd_2007}. They present a tractable formulation of an approximate robust geometric program which by replacing each posynomial with a set of two-term posynomials arrives at the following formulation:
\begin{equation}
\begin{aligned}
& \min &&f_0\left(\vec{x}\right)\\
& \text{s.t.} &&\max_{\vec{\zeta} \in \mathcal{Z}} \left\{e^{\vec{a_{i1}}\left(\zeta\right)\vec{x} + b_{i1}\left(\zeta\right)} + e^{t_1}\right\} &&\leq 1 &&\forall i:K_i \geq 3\\
& &&\max_{\vec{\zeta} \in \mathcal{Z}} \left\{e^{\vec{a_{ik}}\left(\zeta\right)\vec{x} + b_{ik}\left(\zeta\right)} + e^{t_k}\right\} &&\leq e^{t_{k-1}} &&\forall i:K_i \geq 4 \\ & && && && \forall k \in 2,...,K_i-2\\
& &&\max_{\vec{\zeta} \in \mathcal{Z}} \left\{\textstyle{\sum}_{k=K_i-1}^{K_i}e^{\vec{a_{ik}}\left(\zeta\right)\vec{x} + b_{ik}\left(\zeta\right)}\right\} &&\leq e^{t_{K_i-2}} &&\forall i:K_i \geq 3\\
& &&\max_{\vec{\zeta} \in \mathcal{Z}} \left\{\textstyle{\sum}_{k=1}^{K_i}e^{\vec{a_{ik}}\left(\zeta\right)\vec{x} + b_{ik}\left(\zeta\right)}\right\} &&\leq 1 &&\forall i:K_i \leq 2\\
\end{aligned}
\label{Boyd_formulation}
\end{equation}
The left-hand side of each constraint above is either a monomial or a two-term posynomial; monomials are directly tractable, while two-term posynomials can be well approximated by piecewise-linear functions and thereby made tractable, as will be discussed in Section \ref{twoTerm}.\\
As an example, a GP 
\begin{displaymath}
\begin{aligned}
    & \min &&f\\
    & \text{ s.t.} &&\max\left\{M_1 + M_2 + M_3 + M_4\right\} &&\leq 1\\
    & && \max\left\{M_5 + M_6\right\} &&\leq 1
\end{aligned}
\end{displaymath}
would be reformulated as:
\begin{displaymath}
\begin{aligned}
    & \min &&f\\
    & \text{ s.t.} && \max\left\{M_1 + e^{t_1}\right\} && \leq 1\\
    & && \max\left\{M_2 + e^{t_2}\right\} && \leq e^{t_1}\\
    & && \max\left\{M_3 + M_4\right\} && \leq e^{t_2}\\
    & && \max\left\{M_5 + M_6\right\} &&\leq 1
\end{aligned}
\end{displaymath}
Where $\left\{M_i\right\}_{i=1}^{6}$ is a family of monomials.

The formulation in \eqref{Boyd_formulation} decouples every monomial in a posynomial except for the last two, yielding a conservative solution. If $\mathbf{P} = \left\{i:K_i \geq 3\right\}$, $\mathbf{N} = \left\{i:K_i = 2\right\}$, and $\mathbf{M} = \left\{i:K_i = 1\right\}$, and if $r$ is the number of piecewise-linear terms used to approximate a two-term posynomial, then the number of constraints for the approximate tractable problem in \eqref{Boyd_formulation} is: $$r\textstyle{\sum}_{k \in \mathbf{P}} \left(K_i - 1\right) + r|\mathbf{N}| + |\mathbf{M}|$$
\section{Simple Conservative Formulation} \label{Conservative}
This section introduces our first novel approximation of the robust program in \eqref{GP_counterparts_finite}. While this formulation is the simplest to implement, it is also the most conservative as it completely decouples the monomials of each posynomial.

\theoremstyle{definition}
\begin{definition}
A solution $\vec{x}$ to a geometric program is said to be feasible only if $\vec{x}$ satisfies all constraints of that geometric program.
\end{definition}

\begin{definition}
Let $\mathcal{C}$ be a constraint on $\vec{x}$, and let $\mathcal{S}$ be a set of constraints on $\vec{x}$ and some additional variables $\vec{y}$. $\mathcal{S}$ is said to be a safe approximation for $\mathcal{C}$ if the $\vec{x}$ component of every feasible solution $\left(\vec{x},\vec{y}\right)$ is $\mathcal{S}$-feasible to $\mathcal{C}$ \cite{ben-tal_ghaoui_nemirovski_2009}.
\end{definition}
For example, if $f\left(x\right) \leq g\left(x\right)$ $\forall x$, then $g\left(x\right) \leq 1$ is a safe approximation of $f\left(x\right) \leq 1$.

One way to approach the intractability of the dual-form problems in \eqref{GP_counterparts_finite} is to replace each constraint with a tractable safe approximation. The fact that
$$\max_{\vec{\zeta} \in \mathcal{Z}} \textstyle{\sum}_{k=1}^{K_i}e^{\vec{a_{ik}}\left(\zeta\right)\vec{x} + b_{ik}\left(\zeta\right)} \leq \sum_{k=1}^{K_i} {\displaystyle \max_{\vec{\zeta} \in \mathcal{Z}}}\ e^{\vec{a_{ik}}\left(\zeta\right)\vec{x} + b_{ik}\left(\zeta\right)}$$
suggests replacing the intractable constraints by: 
\begin{equation}
	\sum_{k=1}^{K_i}\max_{\vec{\zeta} \in \mathcal{Z}} \left\{e^{\vec{a_{ik}}\left(\zeta\right)\vec{x} + b_{ik}\left(\zeta\right)}\right\} \leq 1 \qquad \forall i \in 1,...,m
	\label{conservative_robust_constraint}
\end{equation}
From \eqref{conservative_robust_constraint}, a suggested safe approximation of  \eqref{GP_counterparts_finite} is:
\begin{equation}
\begin{aligned}
& \min &&f_0\left(\vec{x}\right)\\
& \text{subject to} &&\textstyle{\sum}_{k=1}^{K_i} {\displaystyle \max_{\vec{\zeta} \in \mathcal{Z}}} \left\{e^{\vec{a_{ik}}\left(\zeta\right)\vec{x} + b_{ik}\left(\zeta\right)}\right\} &&\leq 1 &&\forall i \in 1,...,m
\end{aligned}
\label{GP_safe_conservative}
\end{equation}
By adding some dummy variables, the above can be rewritten as:
\begin{equation}
\begin{aligned}
& \min &&f_0\left(\vec{x}\right)\\
& \text{subject to} &&\textstyle{\sum}_{k=1}^{K_i}e^{t_{ik}} &&\leq 1 &&\forall i \in 1,...,m \\
& &&\max_{\vec{\zeta} \in \mathcal{Z}} \left\{e^{\vec{a_{ik}}\left(\zeta\right)\vec{x} + b_{ik}\left(\zeta\right)}\right\} &&\leq e^{t_{ik}} &&\forall i \in 1,...,m &&\forall k \in 1,...,K_i\\
\end{aligned}
\label{GP_safe_decoupled}
\end{equation}
And finally, in convex form:
\begin{equation}
\begin{aligned}
& \min &&\log\left(f_0\left(\vec{x}\right)\right)\\
& \text{subject to} &&\log(\textstyle{\sum}_{k=1}^{K_i}e^{t_{ik}}) &&\leq 0 &&\forall i \in 1,...,m \\
& &&\max_{\vec{\zeta} \in \mathcal{Z}} \left\{\vec{a_{ik}}\left(\zeta\right)\vec{x} + b_{ik}\left(\zeta\right)\right\} &&\leq t_{ik} &&\forall i \in 1,...,m &&\forall k \in 1,...,K_i\\
\end{aligned}
\label{GP_safe_convex}
\end{equation}
Note that dummy variables need not be introduced for constraints with $K_i = 1$.
 
The convex constraints in \eqref{GP_safe_convex} are deprived of data; uncertainty is present only in the linear constraints. As a result, techniques from robust linear programming (Appendix \ref{LP_to_GP}) can be used to robustify the uncertain constraints. Under a polyhedral uncertainty set these robust linear constraints are equivalent to a set of linear constraints, allowing the approximate RGP to be transformed into a GP. A similar transformation can also be done for coefficients in an elliptical uncertainty set.

If the exponents are also in an elliptical uncertainty set, the transformed program is no longer a GP but becomes a cone program with exponential- and second-order-cone (SOC) constraints; programs such as \textit{ECOS} \cite{bib:Domahidi2013ecos} and \textit{SCS} \cite{o’donoghue_chu_parikh_boyd_2016} are used to solve such problems efficiently and reliably.

Note that further research can be done on approximating SOC constraints by GP-compatible softmax functions \cite{hoburg_kirschen_abbeel_2016}, and thus solving the conic program mentioned above as a GP.

\begin{definition}
Two monomials are directly dependent if both depend on the same perturbation element $\zeta_i$. e.g. $M_1 = e^{\vec{a_1}(\zeta_1)\vec{x} + b_1}$ and $M_2 = e^{\vec{a_3}(\zeta_1, \zeta_2)\vec{x} + b_3}$ are directly dependent as they share the perturbation element $\zeta_1$.
\end{definition}

\begin{definition}
Two monomials are indirectly dependent if they are both dependent on a third monomial. e.g. with $M_1 = e^{\vec{a_1}(\zeta_1)\vec{x} + b_1}$, $M_2 = e^{\vec{a_1}(\zeta_2)\vec{x} + b_2}$, and $M_3 = e^{\vec{a_3}(\zeta_1, \zeta_2)\vec{x} + b_3}$, $M_1$ and $M_2$ are indirectly dependent because both share a perturbation element with (and are hence directly dependent on) $M_2$.
\end{definition}

\begin{definition}
Two monomials are independent if they are neither directly nor indirectly dependent.
\end{definition}

\begin{definition}
Two coefficients are consistently dependent if they both increase or both decrease as each perturbation element varies. e.g. if $b_{11} = b_{11}^0 + \zeta_1$ and $b_{12} = b_{12}^0 + 3\zeta_1 - \zeta_2$, then $b_{11}$ and $b_{12}$ are dependent in a consistent way because their coefficients for $\zeta_1$ have the same sign.
\end{definition}

Despite the decoupling of monomials, this formulation is exactly equivalent to \eqref{GP_counterparts_finite} if each posynomial satisfies at least one of the following conditions:
\begin{itemize}
	\item $C_1$: Only one monomial in the posynomial has uncertain parameters
	\item $C_2$: All monomials in the posynomial are either independent or dependent by a shared uncertain monomial factor, e.g. $$
	\begin{aligned}
	p &= e^{a_1\left(\zeta\right)\vec{x}_1 + a_2\vec{x}_2 + b_1\left(\zeta\right)} + e^{a_5\vec{x}_1 + b_3} + e^{a_1\left(\zeta\right)\vec{x}_1 + a_3\vec{x}_3 + b_1\left(\zeta\right)} + e^{a_1\left(\zeta\right)\vec{x}_1 + a_4\vec{x}_2 + b_1\left(\zeta\right)}\\
	&= e^{a_1\left(\zeta\right)\vec{x}_1 + b_1\left(\zeta\right)}\left(e^{a_2\vec{x}_2} + e^{a_3\vec{x}_3} + e^{a_4\vec{x}_2}\right) + e^{a_5\vec{x}_1 + b_3}
	\end{aligned}
	$$
	\item $C_3$: The Perturbation set is independent, e.g. $\mathcal{Z} = \left\{ \vec{\zeta} \in \mathbb{R}^L: \|\vec{\zeta}\|_{\infty} \leq \Gamma\right\}$. Moreover, the monomials in the posynomial are either independent, or their coeffcients are consistently dependent.
\end{itemize}
\ \\
Specifically, if the $i^{th}$ posynomial satisfies $C_1$, $C_2$, or $C_3$ then its maximum under perturbation will always equal the sum of the maximums of each decoupled monomial:
$$
\max_{\vec{\zeta} \in \mathcal{Z}} \left\{\textstyle{\sum}_{k=1}^{K_i}e^{\vec{a_{ik}}\left(\zeta\right)\vec{x} + b_{ik}\left(\zeta\right)}\right\} = \sum_{k=1}^{K_i}\max_{\vec{\zeta} \in \mathcal{Z}} \left\{e^{\vec{a_{ik}}\left(\zeta\right)\vec{x} + b_{ik}\left(\zeta\right)}\right\}
$$

Even when applied to a geometric program whose every posynomial does not satisfy $C_1$, $C_2$, or $C_3$, this formulation has the advantage of requiring only a small number of GP or conic constraints. If $\mathbf{P} = \left\{i:K_i \geq 3\right\}$, $\mathbf{N} = \left\{i:K_i = 2\right\}$, and $\mathbf{M} = \left\{i:K_i = 1\right\}$ then the total number of constraints in \eqref{GP_safe_convex} is: 
$$\textstyle{\sum}_{k \in \mathbf{P}} (K_i+1) + 3|\mathbf{N}| + |\mathbf{M}|$$
which is: 
$$(r-1)\textstyle{\sum}_{k \in \mathbf{P}} (K_i-3) + 2(r-2)|\mathbf{P}| + (r-3)|\mathbf{N}|$$
constraints fewer than in the Two-term formulation of \cite{hsiung_kim_boyd_2007}.
Specifically, this decoupled monomial formulation will always have fewer constraints than \eqref{Boyd_formulation} if 3 or more piecewise sections are used, which will almost always be the case; three linear sections form a poor approximation of a two-term posynomial. Furthermore, this formulation is only more conservative than that in \eqref{Boyd_formulation} if the number of sections is high enough and one of the approximated two-term posynomials captures an active dependence.
\section{Equivalent Intermediate Formulation} \label{EqIntFor} Our focus is now on modifying the methodology of Section \ref{Conservative} to make better approximations of \eqref{GP_counterparts_finite}. This section presents an enhanced formulation of \eqref{GP_safe_convex} that is equivalent to the robust counterparts in \eqref{GP_counterparts_finite} but with smaller, easier to handle posynomial constraints.

In order to divide the posynomial into smaller posynomials, consider the set $\mathbf{I}_i = \left\{ 1,2,...,K_i\right\} $ associated with the $i^{th}$ constraint, and define the equivalence relation $\mathcal{R}$ (see Appendix \ref{EqRel} for the definition of an equivalence relation) given by:
\begin{equation}
\begin{aligned}
\mathcal{R} = \{&k_1 \sim k_2 \iff e^{\vec{a_{ik_1}}\left(\zeta\right)\vec{x} + b_{ik_1}\left(\zeta\right)} \text{ and } e^{\vec{a_{ik_2}}\left(\zeta\right)\vec{x} + b_{ik_2}\left(\zeta\right)} \text{ are dependent} \}
\end{aligned}
\label{equivalence_relation}
\end{equation}

$\mathcal{R}$ splits $\mathbf{I}_i$ into equivalence classes $S_{i,1},\ S_{i,2},\ ...\ S_{i,N_e^i}$, $N_e^i \leq K_i$. Because these classes are completely independent it must be the case that
$$
\max_{\vec{\zeta} \in \mathcal{Z}} \left\{\textstyle{\sum}_{k=1}^{K_i}e^{\vec{a_{ik}}\left(\zeta\right)\vec{x} + b_{ik}\left(\zeta\right)}\right\} = \textstyle{\sum}_{j=1}^{N_e^i} {\displaystyle \max_{\vec{\zeta} \in \mathcal{Z}}} \left\{\textstyle{\sum}_{k \in S_{i,j}}e^{\vec{a_{ik}}\left(\zeta\right)\vec{x} + b_{ik}\left(\zeta\right)}\right\}
$$
The robust counterparts of \eqref{GP_counterparts_finite} are thus equivalent to:
\begin{equation}
\begin{aligned}
&\min &&f_0(x)\\
&\text{s.t.} &&\textstyle{\sum}_{j=1}^{N_e^i} e^{t_{ij}} &&\leq 1 \qquad &&\forall i \in 1,...,m\\
& &&\max_{\vec{\zeta} \in \mathcal{Z}} \left\{\textstyle{\sum}_{k \in S_{i,j}} e^{\vec{a_{ik}}\left(\zeta\right)\vec{x} + b_{ik}\left(\zeta\right)} \right\} &&\leq e^{t_{ij}} &&\forall i \in 1,...,m\\
& && && &&\forall j \in 1, ..., N_e^i\\
\end{aligned}
\label{equivalent_class_setP}
\end{equation}

Our main concern is now the larger posynomials, which motivate the introduction of a \textbf{categorized form} of a robust GP, in which constraints are categorized into three sets. The set of monomial constraints is labeled by $\mathbf{M} \equiv \left\{(i, j): |S_{i,j}| = 1\right\}$, two-term posynomials by $\mathbf{N} \equiv \left\{(i, j): |S_{i,j}| = 2\right\}$, and constraints with three or more posynomials are labeled by $\mathbf{P} \ \equiv \left\{(i,j): |S_{i,j}| \geq 3\right\}$. Further definining $S_{i,j}^k$ as the $k^{th}$ element of $S_{i,j}$, the \textbf{categorized form} is:
\begin{equation}
\begin{aligned}
&\min &\multicolumn{2}{l}{$f_0(x)$}\\
&\text{s.t.} &\multicolumn{2}{l}{$\textstyle{\sum}_{j=1}^{N_e^i} e^{t_{ij}}$} &&\leq 1 \qquad &&\forall i \in 1,...,m\\
& &\max_{\vec{\zeta} \in \mathcal{Z}}& \left\{\textstyle{\sum}_{k \in S_{i,j}} e^{\vec{a_{ik}}\left(\zeta\right)\vec{x} + b_{ik}\left(\zeta\right)} \right\} &&\leq e^{t_{ij}} &&\forall(i,j) \in \mathbf{P}\\
& &\max_{\vec{\zeta} \in \mathcal{Z}}&
\multicolumn{5}{l}{$\left\{e^{\vec{a_{iS_{i,j}^1}}\left(\zeta\right)\vec{x} + b_{iS_{i,j}^1}\left(\zeta\right)} + e^{\vec{a_{iS_{i,j}^2}}\left(\zeta\right)\vec{x} + b_{iS_{i,j}^2}\left(\zeta\right)} \right\} \leq e^{t_{ij}}$}\\
&&&&& &&\forall(i,j) \in \mathbf{N}\\
& &\max_{\vec{\zeta} \in \mathcal{Z}}& \left\{e^{\vec{a_{iS_{i,j}^1}}\left(\zeta\right)\vec{x} + b_{iS_{i,j}^1}\left(\zeta\right)} \right\} &&\leq e^{t_{ij}} &&\forall(i,j) \in \mathbf{M}\\
\end{aligned}
\label{categorizedForm}
\end{equation}
Figure \ref{partitioning} presents an illustrative example of this partitioning. 
\begin{figure}[H]
\captionsetup{justification=centering, font=small}
\begin{center}
\includegraphics[scale=0.25]{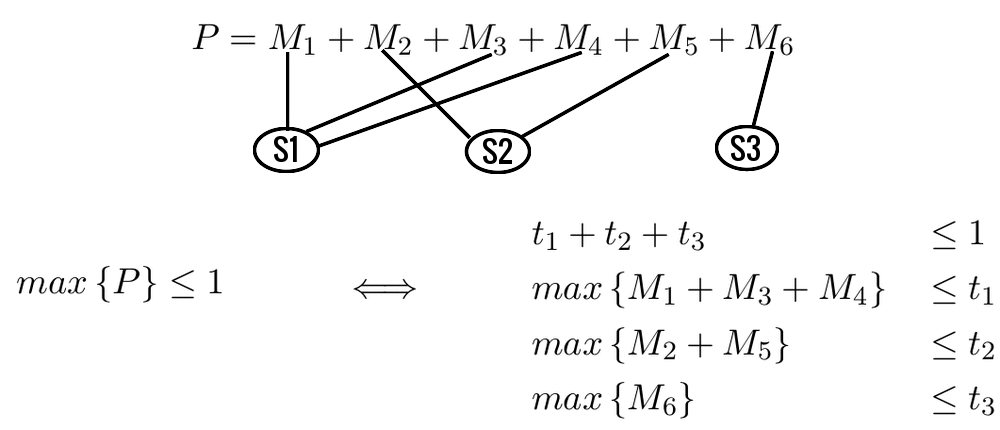}
\end{center}
\caption{Example of partitioning a large posynomial into smaller posynomials.}
\label{partitioning}
\end{figure}

This partitioning will not be of much benefit when all monomials in a posynomial are uncertain and dependent on each other. However, when most monomials are certain, or when some are independent, a large posynomial can be effectively reduced into several smaller ones, or even into monomials. In practice, most GPs for engineering design are only sparsely dependent, and so this formulation provides a significant decrease (often by more than half) in the number of constraints.
\section{Robust Two Term Posynomials[Review]} \label{twoTerm}

We will now review work done in \cite{hsiung_kim_boyd_2007} on using piecewise-linear functions to approximate two-term posynomials. 

Consider the convex function $\phi(x) = \log(1 + e^x)$. The unique best piecewise-linear convex lower approximation of $\phi$  with $r$ terms is defined as:
\begin{equation}
\underline{\phi}_r = 
\begin{cases}
0 \qquad &\text{if} \qquad x \in (- \infty, x_1]\\
\underline{a}_ix + \underline{b}_i &\text{if} \qquad x \in [x_i, x_{i+1}], i=1,2,..,r-2\\
x &\text{if} \qquad x \in [x_{r-1}, \infty)
\end{cases}
\label{lower_phi}
\end{equation}
such that

\begin{align*}
& x_1 < x_2 < ... < x_{r-1} \\
& 0 = \underline{a}_0 < \underline{a}_1 < \underline{a}_2 < ... < \underline{a}_{r-2} < \underline{a}_{r-1} = 1 \\
& 1 = \underline{a}_i + \underline{a}_{r-i-1} \quad\quad \forall i \in \left\{0,1, ..., r-1\right\} \\
& \underline{b}_i = \underline{b}_{r-i-1} \hspace{1.38cm} \forall i \in \left\{1, ..., r-2\right\} \\
& \underline{b}_0 = \underline{b}_{r-1} = 0
\end{align*}

\noindent Moreover, $\exists$ $\tilde{x}_1, \tilde{x}_2, ..., \tilde{x}_{r-2} \in \mathbf{R}$ satisfying
$$
x_1 < \tilde{x}_1 < x_2 < \tilde{x}_2 < ... < x_{r-2} < \tilde{x}_{r-2} < x_{r-1}
$$
such that $\underline{a}_ix + \underline{b}_i$ is tangent to $\phi$ at $\tilde{x}_i$.

The maximum approximation error $\epsilon_r$ of this piecewise-linearization occurs at the break points $x_1, ..., x_{r-1}$ (for a constructive algorithm of the above coefficients, refer to \cite{hsiung_kim_boyd_2007}). This piecewise-linearization can then be used to safely approximate a two-term posynomial.

Letting $h= \log(e^{y_1} + e^{y_2})$ be a two term posynomial in log-space, where $y_1 = \vec{a}_1\vec{x} + b_1$ and $y_2 = \vec{a}_2\vec{x} + b_2$, the unique best r-term piecewise-linear lower approximation is: 
\begin{equation}
\begin{aligned}
\underline{h_r} = \max \{&\underline{a}_{r-1}y_1 + \underline{a}_0y_2 + \underline{b}_0, \underline{a}_{r-2}y_1 + \underline{a}_1y_2 + \underline{b}_1, \underline{a}_{r-3}y_1 + \underline{a}_2y_2 + \underline{b}_2,\ ...,\\
 & \underline{a}_{1}y_1 + \underline{a}_{r-2}y_2 + \underline{b}_{r-2}, \underline{a}_0y_1 + \underline{a}_{r-1}y_2 + \underline{b}_{r-1}\}
\end{aligned}
\end{equation}
while its unique best r-term piecewise-linear upper approximation is: 
\begin{equation}
\overline{h_r} = \underline{h_r} + \epsilon_r
\end{equation}
where $\underline{a}_{0}, \underline{a}_{1}, ..., \underline{a}_{r-1}$ and $\underline{b}_{0}, \underline{b}_{1}, ..., \underline{b}_{r-1}$ are as given in equation \eqref{lower_phi}, and $\epsilon_r$ is the maximum error between $\phi$ and $\underline{\phi}_r$.

Since $\overline{h}_r \geq h$, then each posynomial in the set $\mathbf{N}$ can be safely approximated by its own $\overline{h_r}$. Replacing the two term posynomial constraints by their piecewise-linear lower approximation will lead to a relaxed formulation, and thus the difference between the safe formulation's solution and the relaxed formulation's solution is an indication of how good an approximation is, and whether the number of piecewise-linear terms $r$ should be further increased or not.

Because a piecewise-linear constraint can be represented as a set of linear constraints, these two-term posynomial approximations can be transformed to GP. This GP can be made less conservative by increasing the number of piecewise-linear terms.
\section{Robust Large Posynomials} \label{k_term}
Having introduced improved techniques for monomials and two-term posynomials, we now turn to better approximations for posynomial constraints with more than two terms (those associated with $\mathbf{P}$). Two novel methodologies for approximating large posynomials will be presented: the first is only able to transform posynomials whose coefficients are uncertain, while the second can transform posynomials whose coefficients and exponents are uncertain.

\subsection{Linearized Perturbations Formulation} \label{uncertain_coeff}
In the majority of engineering-design constraints, exponents are derived with certainty from physical laws or dimensional analysis, and so programs where uncertainty is only present in the coefficients are common. Let

\begin{displaymath}
\vec{b} = \vec{b}^0 + \textstyle{\sum}_{l=1}^{L}\vec{b}^l\zeta_l
\end{displaymath}
where $\vec{\zeta} \in \mathcal{Z}$ is as given by equation \eqref{perturbation_set}.
The second set of constraints in equation \eqref{categorizedForm} is then equivalent to: 

$$
\begin{aligned}
&\max_{\vec{\zeta} \in \mathcal{Z}} \left\{\textstyle{\sum}_{k \in S_{i,j}} e^{\vec{a_{ik}}\vec{x} + b^0_{ik}}e^{\textstyle{\sum}_{l=1}^{L}b^l_{ik}\zeta_l} \right\} &&\leq e^{t_{ij}} &&\forall (i, j) \in \mathbf{P}\\
\Leftrightarrow &\max_{\vec{\zeta} \in \mathcal{Z}} \left\{\textstyle{\sum}_{k \in S_{i,j}}\left(\textstyle{\prod}_{l=1}^{L}e^{b^l_{ik}\zeta_l}\right) e^{\vec{a_{ik}}\vec{x} + b^0_{ik}} \right\} &&\leq e^{t_{ij}} &&\forall (i, j) \in \mathbf{P}\\
\end{aligned}
$$
Thus by applying the following change of variable
$$v_{i,j}^k = e^{\vec{a_{ik}}\vec{x} + b^0_{ik}} \qquad \forall k \in S_{i,j}$$
the constraint above can be rewritten as:
\begin{equation}
\max_{\vec{\zeta} \in \mathcal{Z}} \left\{\textstyle{\sum}_{k \in S_{i,j}}\left(\textstyle{\prod}_{l=1}^{L}e^{b^l_{ik}\zeta_l}\right) v_{i,j}^k \right\} \leq e^{t_{ij}} \qquad \forall (i, j) \in \mathbf{P}
\label{linearCon_expPerts}
\end{equation}

Although the constraints in \eqref{linearCon_expPerts} are linear in terms of $\vec{v_{i,j}}$, the perturbations are not affine but exponential, and should be linearized to improve tractability.
 
\subsubsection{Linearizing Perturbations}
The convexity of exponential perturbations (see Figure \ref{perts_convexity}) implies that there exists some half-space $[\vec{f}_{i,j}^k]^T\vec{\zeta} + g_{i,j}^k$ such that 
$$[\vec{f}_{i,j}^k]^T\vec{\zeta} + g_{i,j}^k \geq \textstyle{\prod}_{l=1}^{L}e^{b^l_{ik}\zeta_l} \qquad \forall (i, j) \in \mathbf{P} \qquad \forall k \in S_{i,j}$$
on some finite domain for $\vec{\zeta}$.

\begin{figure}[H]
\captionsetup{justification=centering, font=small}
\begin{center}
\includegraphics[scale=0.7]{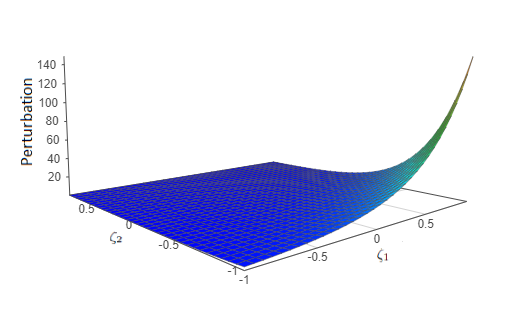}
\end{center}
\caption{Example of the convexity of exponential perturbations.}
\label{perts_convexity}
\end{figure}
 
Accordingly, a possible safe approximation of the constraints in equation \eqref{linearCon_expPerts} is:
\begin{equation}
\max_{\vec{\zeta} \in \mathcal{Z}} \left\{\textstyle{\sum}_{k \in S_{i,j}}\left([\vec{f}_{i,j}^k]^T\vec{\zeta}\right)v_{i,j}^k \right\} + \textstyle{\sum}_{k \in S_{i,j}}g_{i,j}^k v_{i,j}^k \leq e^{t_{ij}} \qquad \forall (i, j) \in \mathbf{P}
\label{linearCon_linPerts}
\end{equation}

To construct the half-space $[\vec{f}_{i,j}^k]^T\vec{\zeta} + g_{i,j}^k$, and taking into consideration the fact that $-1 \leq \zeta_l \leq 1$ for $l = 1,...,L$, we suggest the following algorithm
\\
\begin{enumerate}
	\item Find the list of vertices $\mathcal{V}$ of the unit box in $\mathbf{R}^L$
	\item Find the list of values $\mathcal{O}$ of $\textstyle{\prod}_{l=1}^{L}e^{b^l_{ik}\zeta_l}$ at the vertices $\mathcal{V}$, note that the $i^{th}$ vertex corresponds to the $i^{th}$ value.
	\item Find the maximum $M_k$ and minimum $m_k$ of $\mathcal{O}$ and their corresponding vertices $\vec{\zeta}_M$ and $\vec{\zeta}_m$.
	\item Solve the least-squares problem 
	$$\min_{\vec{f}_{i,j}^k,\ g_{i,j}^k} \sqrt{\textstyle{\sum}_{\alpha=1}^{|\mathcal{O}|}([\vec{f}_{i,j}^k]^T\mathcal{V}_{\alpha} + g_{i,j}^k - \mathcal{O}_{\alpha})^2}$$
	such that $[\vec{f}_{i,j}^k]^T\vec{\zeta}_M + g_{i,j}^k = M_k$ and $[\vec{f}_{i,j}^k]^T\vec{\zeta}_m + g_{i,j}^k = m_k$
\end{enumerate}
\ \\
Figure \ref{perts_convexity} shows a 2D example where the half-space passes through the highest point (1, -1) and the lowest point (-1, 1).

This half-space is safe for all types of uncertainty sets (normalized to fit into a unit box), although more-specific half-spaces can be constructed for a specific uncertainty set. For instance, with elliptical uncertainty sets, the vertices in the first step might be replaced by an equidistant set of samples from the surface of the uncertainty set.

\subsubsection{Signomial Programming Compatible Constraint}
Equation \eqref{linearCon_linPerts} can be transformed with robust linear programming techniques. Unfortunately, because some components of $\vec{f}_{i,j}^k$ might not be positive, the resulting set of robust constraints is not always GP compatible. Appendix \ref{sigProg} defines signomial programming and clarifies how it is useful to our discussion; in short, it allows the specification and efficient solving of a difference-of-convex program with these negative coefficients.

Signomial problems generally do not need a starting point to solve, but specifying one can speed convergence, and in robust programming there are two quick and obvious candidates for a starting point: the nominal solution of the original program, or the conservative solution of the GP-compatible formulation in Section \ref{Conservative}.

The number of additional possibly SP constraints needed per uncertain posynomial is equal to the number of uncertain parameters in that posynomial. Decoupling the monomials in each posynomial is equivalent to evaluating the exponential perturbations at its largest value, however, in \eqref{linearCon_linPerts}, the perturbation will be evaluated on a lower point on the half space. As a result, this formulation is guaranteed to be less conservative than that of Section \ref{Conservative}, or the same if both are exact.

\subsection{Best Pairs Formulation}\label{uncertain_exps}
We now discuss a framework for robustifying constraints with uncertain coefficients and exponents. Consider the constraints in $\mathbf{P}$
\begin{equation}
\max_{\vec{\zeta} \in \mathcal{Z}} \left\{\textstyle{\sum}_{k \in S_{i,j}} e^{\vec{a_{ik}}\left(\zeta\right)\vec{x} + b_{ik}\left(\zeta\right)} \right\} \leq e^{t_{ij}} \qquad \forall (i, j) \in \mathbf{P}
\label{second_set}
\end{equation}
Let $\mathcal{P}_{i,j}$ be the set of permutations on $S_{i,j}$, and let $\phi_k$ be the image of an element $k \in S_{i,j}$ under the permutation $\phi \in \mathcal{P}_{i,j}$. Accordingly, \eqref{second_set} is equivalent to: 
\begin{equation}
\max_{\vec{\zeta} \in \mathcal{Z}} \left\{\textstyle{\sum}_{k \in S_{i,j}} e^{\vec{a_{i\phi_k}}\left(\zeta\right)\vec{x} + b_{i\phi_k}\left(\zeta\right)} \right\} \leq e^{t_{ij}} \qquad \forall (i, j) \in \mathbf{P} \qquad \phi \in \mathcal{P}_{i,j}
\label{permutation_k_term}
\end{equation}

Our goal is to maximize pairs of monomials while trying to choose the least conservative combination. Let 
$$
\vec{a_{iS_{i,j}^k}}\left(\zeta\right)\vec{x} + b_{iS_{i,j}^k}\left(\zeta\right) = \mathcal{L}^k_{i,j}\quad \forall (i,j) \in \mathbf{P}
$$
where $S_{i,j}^k$ is the $k^{th}$ element of $S_{i,j}$, and assume that $|S_{i,j}|$ is even, such that
$$
\displaystyle \max_{\vec{\zeta} \in \mathcal{Z}} \left\{\textstyle{\sum}_{k=1}^{|S_{i,j}|} e^{\mathcal{L}^{\phi_k}_{i,j}}\right\} \leq \textstyle{\sum}_{k=1}^{|S_{i,j}|/2} \displaystyle \max_{\vec{\zeta} \in \mathcal{Z}} \left\{e^{\mathcal{L}^{\phi_{2k-1}}_{i,j}} + e^{\mathcal{L}^{\phi_{2k}}_{i,j}}\right\}
$$
for any permutation $\phi \in \mathcal{P}_{i,j}$. Therefore
$$
\displaystyle \max_{\vec{\zeta} \in \mathcal{Z}} \left\{\textstyle{\sum}_{k=1}^{|S_{i,j}|} e^{\mathcal{L}^{\phi_k}_{i,j}}\right\} \leq \min_{\phi \in \mathcal{P}_{i,j}}\left\{\textstyle{\sum}_{k=1}^{|S_{i,j}|/2} \displaystyle \max_{\vec{\zeta} \in \mathcal{Z}} \left\{e^{\mathcal{L}^{\phi_{2k-1}}_{i,j}} + e^{\mathcal{L}^{\phi_{2k}}_{i,j}}\right\}\right\}
$$
and
\begin{equation}
\min_{\phi \in \mathcal{P}_{i,j}}\left\{\textstyle{\sum}_{k=1}^{|S_{i,j}|/2} \displaystyle \max_{\vec{\zeta} \in \mathcal{Z}} \left\{e^{\mathcal{L}^{\phi_{2k-1}}_{i,j}} + e^{\mathcal{L}^{\phi_{2k}}_{i,j}}\right\}\right\} \leq e^{t_{ij}} \qquad \forall (i, j) \in \mathbf{P}
\label{min_safe_app}
\end{equation}
is a safe approximation of \eqref{permutation_k_term}. From the fact that
$$
\max(a + b) + \max(c + d) = \max(c + d) +  \max(a + b) = \max(b + a) + \max(d + c)
$$
we see that some perturbations result in identical safe approximations. The permutation set can be modified to contain differing permutations only; let $\mathcal{P}_{i,j}$ represent the set of differing permutations. The number of differing permutations is given by:
\begin{equation}
|\mathcal{P}_{i,j}| = \frac{\binom{|S_{i,j}|}{2}\binom{|S_{i,j}|-2}{2}\binom{|S_{i,j}|-4}{2}...\binom{4}{2}}{(|S_{i,j}|/2)!}
\label{no_pert_even}
\end{equation}
If the number of monomial terms for a given posynomial is large, then $|\mathcal{P}_{i,j}|$ might become quite large. Therefore, we might choose to work with a subset $\hat{\mathcal{P}}_{i,j}$ of $\mathcal{P}_{i,j}$, where the permutations are either chosen depending on the structure of the posynomial or randomly selected. Constraints associated with $\mathbf{P}$ will be replaced by: 
\begin{equation}
\textstyle{\sum}_{k=1}^{|S_{i,j}|/2} {\displaystyle \max_{\vec{\zeta} \in \mathcal{Z}}} \left\{e^{\mathcal{L}^{\phi_{2k-1}}_{i,j}} + e^{\mathcal{L}^{\phi_{2k}}_{i,j}}\right\}\leq e^{t_{ij}} \qquad \forall (i, j) \in \mathbf{P} \qquad \phi \in \hat{\mathcal{P}}_{i,j}
\label{two_term_max_even}
\end{equation}
The constraints in \eqref{two_term_max_even} are safe approximations for those in \eqref{permutation_k_term}, and equivalent to:
\begin{equation}
\begin{aligned}
&\textstyle{\sum}_{k=1}^{|S_{i,j}|/2} e^{z_{ij}^k} &&\leq e^{t_{ij}} \qquad &&\forall (i, j) \in \mathbf{P}\\
&\max_{\vec{\zeta} \in \mathcal{Z}} \left\{e^{\mathcal{L}^{\phi_{2k-1}}_{i,j}} + e^{\mathcal{L}^{\phi_{2k}}_{i,j}}\right\} &&\leq e^{z_{ij}^k} &&\forall (i, j) \in \mathbf{P} \qquad \forall k \in \left\{1,..,|S_{i,j}|/2\right\}
\end{aligned}
\label{two_term_even}
\end{equation}
Accordingly, the robust counterparts in Section \ref{RGP} will be safely approximated by:
\begin{equation}
\begin{aligned}
&\min &&f_0(x)\\
&\text{s.t.} &&\textstyle{\sum}_{j=1}^{N_e^i} e^{t_{ij}} &&\leq 1 \quad &&\forall i \in 1,..,m\\
& &&\max_{\vec{\zeta} \in \mathcal{Z}} \left\{e^{\mathcal{L}_{i,j}^1} \right\} &&\leq e^{t_{ij}} &&\forall (i,j) \in \mathbf{M}\\
& &&\max_{\vec{\zeta} \in \mathcal{Z}} \left\{e^{\mathcal{L}_{i,j}^1} + e^{\mathcal{L}_{i,j}^2} \right\} &&\leq e^{t_{ij}} &&\forall (i,j) \in \mathbf{N}\\
& &&\textstyle{\sum}_{k=1}^{|S_{i,j}|/2} e^{z_{ij}^k} &&\leq e^{t_{ij}} \quad &&\forall (i,j) \in \mathbf{P}\\
& &&\max_{\zeta \in \mathcal{Z}} \left\{e^{\mathcal{L}^{\phi_{2k-1}}_{i,j}} + e^{\mathcal{L}^{\phi_{2k}}_{i,j}}\right\} &&\leq e^{z_{ij}^k} &&\forall (i,j) \in \mathbf{P}\\
& && && &&\forall k \in \left\{1,..,|S_{i,j}|/2\right\}\\
& && && && \phi \in \hat{\mathcal{P}}_{i,j}\\
\end{aligned}
\label{GP_counterparts_tractable}
\end{equation}

Each constraint in \eqref{GP_counterparts_tractable} is composed of monomials, two-term posynomials, or data-deprived large posynomials, and so the program is tractable.

Our final step is finding good permutations, i.e. permutations that will make our solution less conservative. For that sake, the following lemma will be utilized

\begin{lemma}
Considering the two optimization problems 
$$
\begin{aligned}
&\min f(\vec{x})\\
&\text{s.t. } \mathcal{S}_i(\vec{x}) \leq 0 \qquad i = 1,2,...,n
\end{aligned}
$$
and
$$
\begin{aligned}
&\min f(\vec{x})\\
&\text{s.t. } \mathcal{T}_i(\vec{x}) \leq 0 \qquad i = 1,2,...,n
\end{aligned}
$$
let $\vec{x}_1$ and $\vec{x}_2$ be their respective solutions. If $\mathcal{T}_i(\vec{x}_1) \leq \mathcal{S}_i(\vec{x}_1) \ \ \forall i \in \left\{1,2,...,n\right\}$, then $f(\vec{x}_2) \leq f(\vec{x}_1)$
\label{the_lemma}
\end{lemma}
\begin{proof}
$$\mathcal{T}_i(\vec{x}_1) \leq \mathcal{S}_i(\vec{x}_1) \leq 0 \quad \forall i \in \left\{1,2,...,n\right\} $$ 
So $\vec{x}_1$ is a feasible solution for the second optimization problem. But since $\vec{x}_2$ is the optimal solution of the second problem, then
$$f(\vec{x}_2) \leq f(\vec{x}_{feasible}) \implies f(\vec{x}_2) \leq f(\vec{x}_1)$$
\ 
\end{proof}

Starting from the above lemma, the algorithm below illustrates how relatively good permutations might be chosen:
\\
\begin{enumerate}
\item Randomly choose the permutations $\phi$ from the set $\hat{\mathcal{P}}_{i,j}$ $\forall (i,j) \in \mathbf{P}$
\item Solve the optimization problem in \eqref{GP_counterparts_tractable}, and let $\vec{x}_1$ be the solution
\item Repeat
\begin{enumerate}
\item $\forall (i,j) \in \mathbf{P}$, select the new permutations $\phi \in \hat{\mathcal{P}}_{i,j}$ such that $\phi$ minimizes $\textstyle{\sum}_{k=1}^{|S_{i,j}|/2} {\displaystyle \max_{\vec{\zeta} \in \mathcal{Z}}} \left\{e^{\mathcal{L}^{\phi_{2k-1}}_{i,j}} + e^{\mathcal{L}^{\phi_{2k}}_{i,j}}\right\}\bigg\rvert_{\vec{x}_i}$
\item Solve the optimization problem in \eqref{GP_counterparts_tractable}, and let $\vec{x}_i$ be the solution
\item If $\vec{x}_i = \vec{x}_{i-1}$ : break
\end{enumerate}
\end{enumerate}
\ \\
Although the final solution might not be the globally least-conservative over all possible permutations it is guaranteed to be the locally least-conservative given the descent algorithm based on Lemma \eqref{the_lemma}. The Best Pairs Formulation is guaranteed to be less-conservative than the Two-term formulation because it involves less monomial decoupling and fewer piecewise-linear approximations.
\section{Applications} \label{apps}
In this section, uncertain design problems are transformed into robust programs to demonstrate and compare the methods above. The first problem (a simple aircraft model) is used mostly to clarify the concept, while the second (a large-scale solar-aircraft design) shows the effectiveness of the tractable robust GP-approximation methods proposed.

\subsection{Simple Flight Design}
Hoburg describes each of the sub-models constituting this simple wing design in \cite{hoburg_abbeel_2014}, the resulting GP model is as follows:

\begin{equation}
\begin{aligned}
	&\min &&D\\
	& s.t. &&C_D \geq \frac{(CDA0)}{S} + kC_fS_{wet} + \frac{C_L^2}{\pi Ae}\\
	& &&W_W \geq W_{W_2} S + \frac{W_{W_1}N_{ult}A^{1.5}\sqrt{W_0WS}}{\tau}\\
	& &&D \geq 0.5\rho SC_DV^2\\
	& &&Re \leq \frac{\rho}{\mu}V\sqrt{\frac{S}{A}}\\
	& &&C_f \geq \frac{0.074}{Re^{0.2}}\\
	& &&W \leq 0.5\rho S C_LV^2\\
	& &&W \leq 0.5\rho S C_{Lmax}V_{min}^2\\
	& &&W \geq W_0 + W_W
\end{aligned}
\label{simple_wing}
\end{equation}
\ \\
This 8-constraint problem has two design variables ($A, S$), seven dependent free variables ($C_D, C_L, C_f, Re, W, W_W, V$), and thirteen uncertain parameters detailed in Table \ref{uncertain_simple_wing}. Note that the uncertainties specified in the table are the principal-axis widths of either the box or elliptical uncertainty sets and that the uncertainties in this problem are only in the coefficients. The constraints in \eqref{simple_wing} are predominantly monomials and two-term posynomials, with only one posynomial of more than two terms.
\ \\
\begin{center}
\begin{tabular}{ |P{7em}|P{3.5cm}|P{5.2cm}|}
\hline
Uncertain Parameter & Value & Description \\
\hline
$(CDA0)$ & $0.0350 $ $[m^2]$ $\pm 42.8\%$ & Fuselage Drag Area\\
\hline
$k$ & $1.170 \pm 31.1\%$ & Form Factor \\
\hline
$S_{wet}$ & $2.075 \pm 3.61\%$ & Wetted Area Ratio\\  
\hline
$e$ & $0.9200 \pm 7.60 \%$ & Oswald Efficiency Factor\\  
\hline
$W_{W_2}$ & $60.00 $ $[Pa]$ $\pm 66.0\%$ & Wing Weight Coefficient 2\\  
\hline
$W_{W_1}$ & $12.00e^{-5} $ $[\frac{1}{m}]$ $\pm 60.0\%$ & Wing Weight Coefficient 1\\  
\hline
$N_{ult}$ & $3.300 \pm 33.3\%$ & Ultimate Load Factor\\  
\hline
$W_0$ & $6250$ $[N]$ $\pm 60.0\%$ & Aircraft Weight Excluding Wing\\  
\hline
$\tau$ & $0.1200 \pm 33.3\%$ & Airfoil Thickness to Chord Ratio\\  
\hline
$\rho$ & $1.230$ $[\frac{kg}{m^3}]$ $\pm 10.0\%$ & Density of Air\\  
\hline
$\mu$ & $1.775e^{-5}$ $[\frac{kg}{ms}]$ $ \pm 4.22\%$ & Viscosity of Air\\  
\hline
$C_{Lmax}$ & $1.600 \pm 25.0\%$ & Maximum Lift coefficient of Wing\\  
\hline
$V_{min}$ & $25.00$ $[\frac{m}{s}]$ $ \pm 20.0\%$ & Takeoff Speed\\  
\hline
\end{tabular}
\captionof{table}{Uncertain parameters in the simple wing design.}
\label{uncertain_simple_wing}
\end{center}

The problem is first solved for different sizes ($\Gamma$ values) of box and elliptical uncertainty sets (defined in Appendix \ref{LP_to_GP}), where the number of piecewise-linear terms is chosen such that the relative error between the solution of the upper tractable approximation and that of the lower tractable approximation is less than $0.1\%$. The design variables are then fixed for each solution so that the design can be simulated for 1000 realizations of the uncertain parameters to examine average design performance.\\

\begin{figure}[ht]
    \centering
    \captionsetup{justification=centering, font=small}
    \begin{subfigure}{0.49\textwidth}
        \centering
        \includegraphics[height=1.85in]{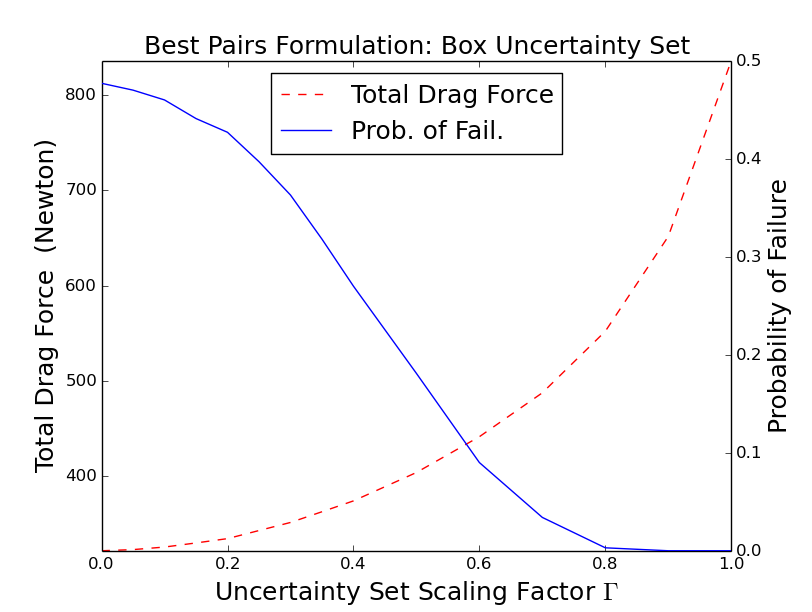}
    \end{subfigure}%
    ~ 
    \begin{subfigure}{0.49\textwidth}
        \centering
        \includegraphics[height=1.85in]{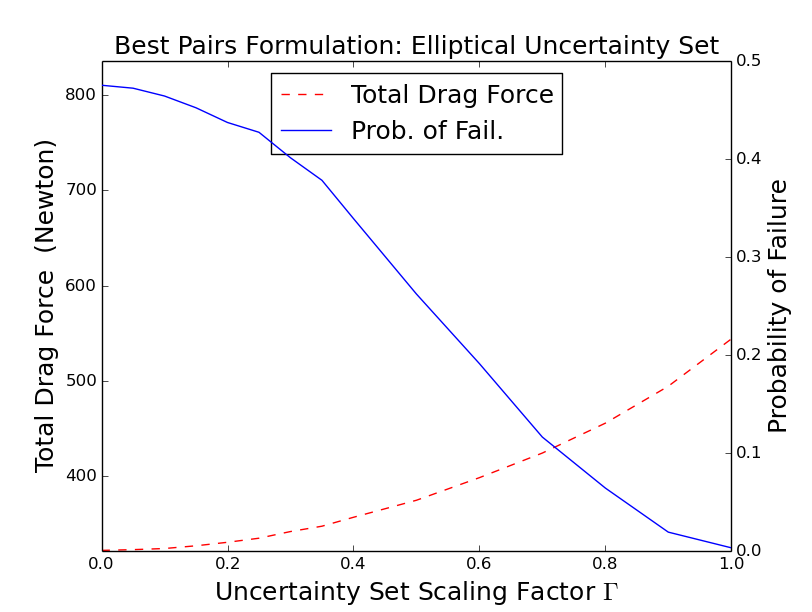}
    \end{subfigure}
    \caption{Performance of the optimal robust simple wing, using the Best Pairs formulation, as a function of $\Gamma$ for different uncertainty sets.}
    \label{simple_wing_var_gamma}
\end{figure}

We can see from Figure \ref{simple_wing_var_gamma} that probability of failure goes to zero as $\Gamma$ increases. Descriptions of how $\Gamma$ determines the size of uncertainty sets are in Appendices \ref{box_linear_app} and \ref{ellip}. For a given level of risk, the average performance of designs robust to an elliptical uncertainty set is better than for those using a box uncertainty set. In other words, modeling the uncertainty using an elliptical uncertainty set would guarantee a safe design and better performance when compared to using a box uncertainty set. Moreover, using margins would in the best case be as good as using a box uncertainty set, and therefore will lead in an inferior performance.

Table \ref{simp_table} and Figure \ref{compare_simple_wing} compare the results of robustifying and solving the simple wing design problem with different methodologies. The number of piecewise-linear terms per two-term posynomial needed to reach a $0.1\%$ tolerance is much higher for the Two Term formulation from \cite{hsiung_kim_boyd_2007} than for our formulations. Correspondingly, the total number of constraints needed by the two-term formulation is also much larger.\\

\begin{center}
\begin{tabular}{ |m{7em}|m{1.8cm}|m{1.5cm}|m{2.2cm}|m{1.8cm}| }
\hline
Method & Uncertainty Set & Relative Error [$\%$] & Number of PWL Sections & Number of Constraints \\
\hline
\multirow{2}{7em}{Two Term} & Box & 0.1 & 79 & 327\\ 
& Elliptical & 0.1 & 70 & 291\\
\hline
\multirow{2}{7em}{Simple Conservative} & Box & 0 & 0 & 16\\ 
& Elliptical & 0 & 0 & 16 \\
\hline
\multirow{2}{7em}{Linearized Perturbations} & Box & 0 & 0 & 16\\ 
& Elliptical & 0.1 & 31 & 49 \\  
\hline
\multirow{2}{7em}{Best Pairs} & Box & 0 & 0 & 16\\ 
& Elliptical & 0.1 & 40 & 93 \\
\hline
Deterministic & N/A & N/A & N/A & 8\\
\hline
\end{tabular}
\captionof{table}{Comparing the simple wing results using different methodologies.}
\label{simp_table}
\end{center}

\begin{figure}[ht]
    \centering
    \captionsetup{justification=centering, font=small}
    \begin{subfigure}{0.499\textwidth}
        \centering
        \includegraphics[height=1.85in]{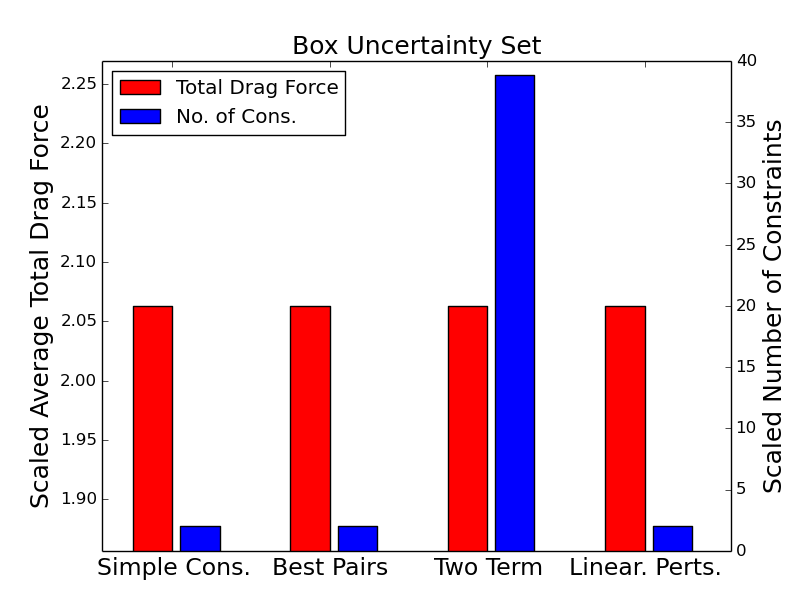}
    \end{subfigure}%
    ~ 
    \begin{subfigure}{0.49\textwidth}
        \centering
        \includegraphics[height=1.85in]{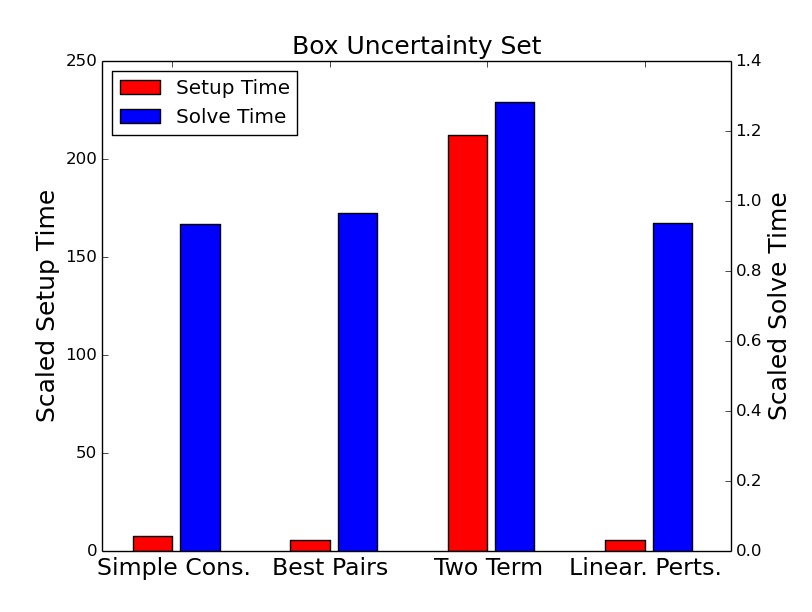}
    \end{subfigure}
    ~
    \begin{subfigure}{0.499\textwidth}
        \centering
        \includegraphics[height=1.85in]{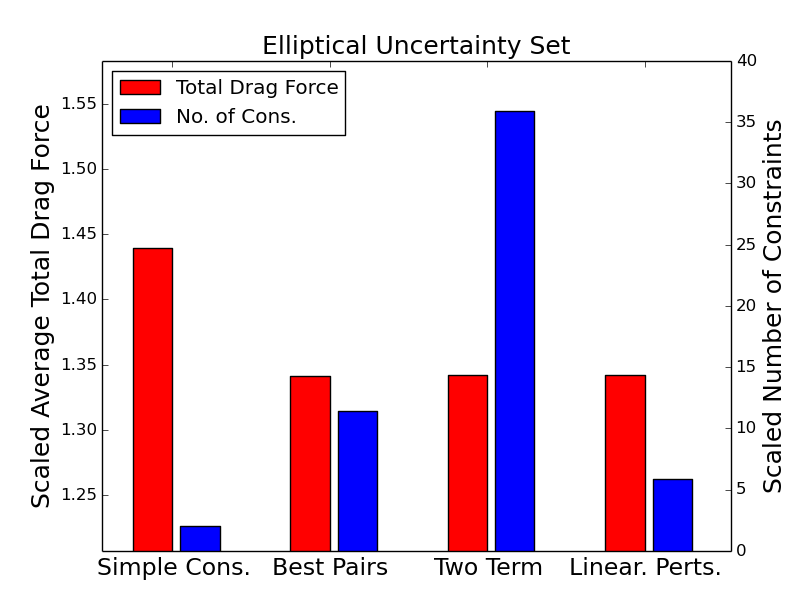}
    \end{subfigure}%
    ~ 
    \begin{subfigure}{0.49\textwidth}
        \centering
        \includegraphics[height=1.85in]{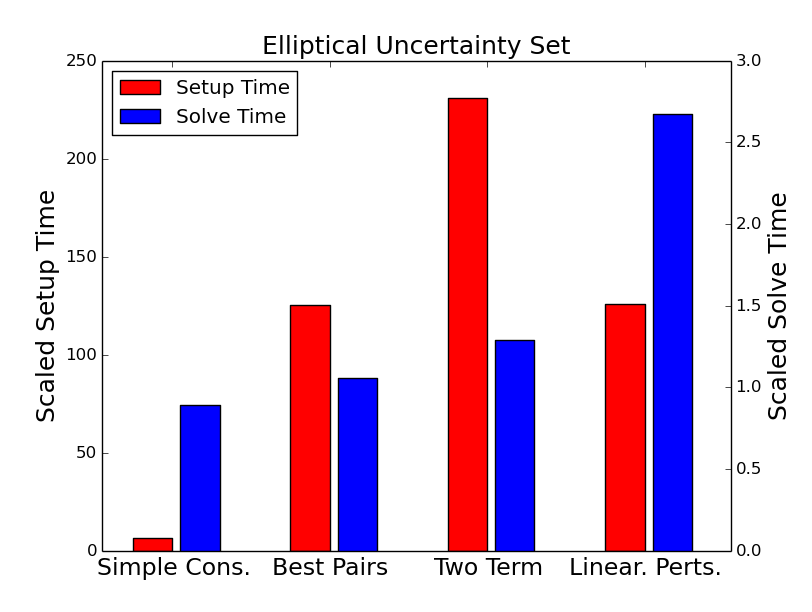}
    \end{subfigure}
    \caption{Robust simple wing design results relative to the nominal design problem.}
    \label{compare_simple_wing}
\end{figure}

Our methodologies achieved an exact solution for the box uncertainty set as condition $C_3$ (Section \ref{Conservative}) is satisfied, and therefore no piecewise-linearization is needed. This explains the identical results of our three methodologies for the box uncertainty set in Figures \ref{compare_simple_wing} and \ref{simple_wing_var_pwl}, constrasting with the more than 70 piecewise-linear sections needed by the Two-term formulation to achieve a $0.1\%$ tolerance. Condition $C_3$ is not satisfied for the elliptical uncertainty set, and so here our simple method is noticeably more conservative than those using piecewise-linear approximations.

\begin{figure}[ht]
    \centering
    \captionsetup{justification=centering, font=small}
    \begin{subfigure}{0.49\textwidth}
        \centering
        \includegraphics[height=1.85in]{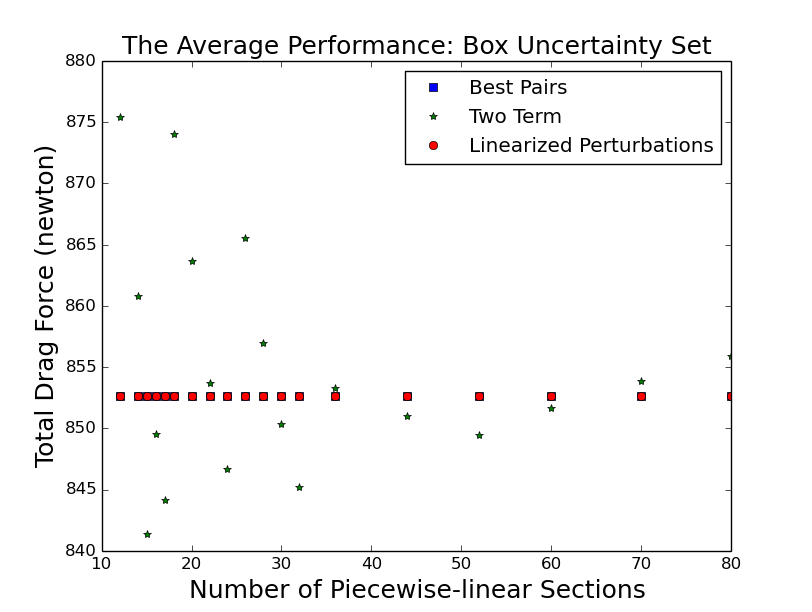}
    \end{subfigure}%
    \begin{subfigure}{0.49\textwidth}
        \centering
        \includegraphics[height=1.85in]{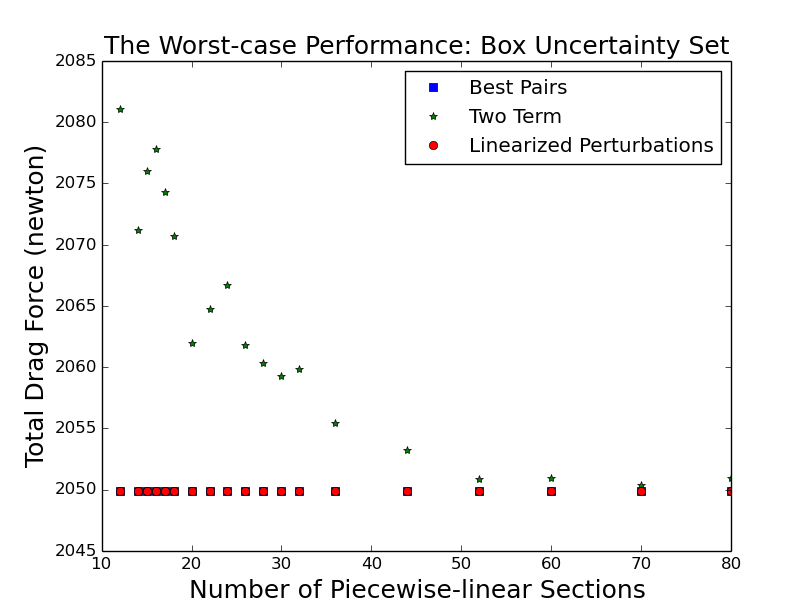}
    \end{subfigure}
    \begin{subfigure}{0.49\textwidth}
        \centering
        \includegraphics[height=1.85in]{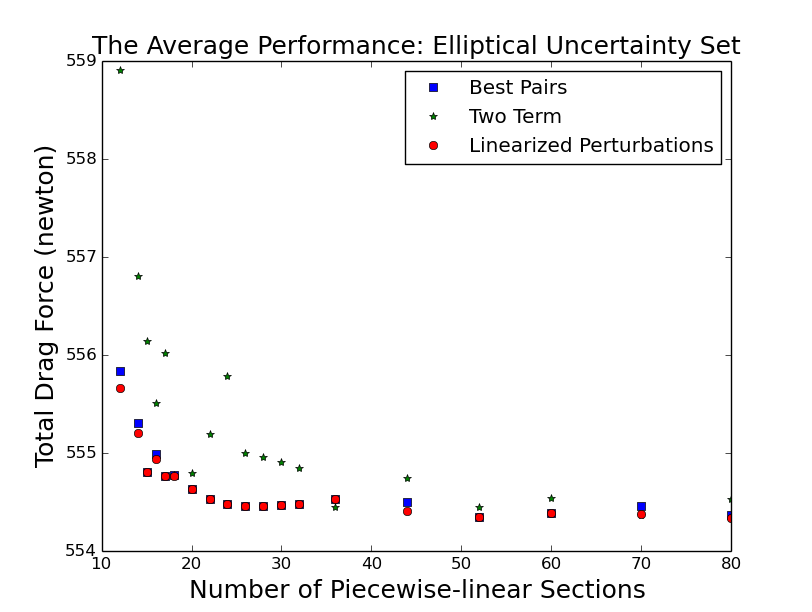}
    \end{subfigure}%
    \begin{subfigure}{0.49\textwidth}
        \centering
        \includegraphics[height=1.85in]{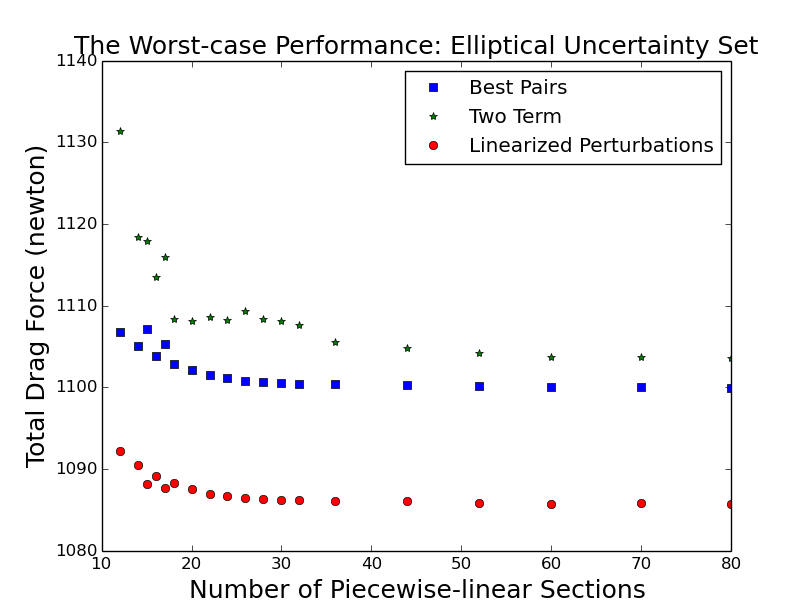}
    \end{subfigure}
    \caption{The performance of the simple wing design as a function of the number of piecewise-linear terms for different methodologies.}  
    \label{simple_wing_var_pwl}
\end{figure}

\subsection{UAV Solar Model}
In previous GP modeling work, Burton presented a GP-compatible solar-powered aircraft model \cite{burton_hoburg} designed to fulfill the requirements in Table \ref{mission_requirement}. Sizing of such long-endurance aircraft is complicated because of the multifaceted interaction between aerodynamics, structural weight, solar energy, wind speed, and other disciplines and environmental models. For this study, Burton chose a fixed-wing tractor configuration as a baseline for the solar aircraft architecture with constant tapered wing and a conventional tail with a single tail boom extending from the wing. The batteries responsible for storing the solar energy are held in the wings while the solar cells are placed along the wings and possibly on the horizontal tail as shown in the simple diagram in Figure \ref{the_aircraft}. To evaluate the size and performance of this aircraft, the coupled environmental, structural, and performance models are expressed in a geometric programming form and combined to form an optimization problem that is guaranteed to converge to the global optimum in a fraction of second, where such a speed has not previously been achieved in conceptual sizing studies for Solar UAV problem attempted here. More details about the different models and the effect of each model on the overall design can be found in \cite{burton_hoburg}.

Burton studied the feasibility limits of the solar aircraft by introducing margins to account for uncertainties such as in the wind speed and solar cells efficiency. We proposea better feasibility study by using robust geometric programming. The deterministic model is formed of 2698 constraints including 20 design variables, 875 dependent free variables, and 271 uncertain parameters. Table \ref{uncertain_solar} presents a sample of the uncertain parameters in the solar aircraft design.\\

\begin{table}[ht]
\begin{minipage}[b]{0.38\linewidth}
\centering
\begin{tabu}{P{6em}|P{1.9cm}}
    \tabucline[1pt]{1-6}
    \multicolumn{2}{c}{\textbf{Mission Requirements}} \\ \tabucline[1pt]{1-6}
    Payload & $10$ $lbs$ \\ \hline
    Station Keeping & $90\%$ $winds$\\ \hline
    Endurance & $> 5$ $days$\\ \hline
    Season & $all$ $seasons$\\ \hline
    Altitude & $> 4600$ $m$\\ \hline
    Latitude & $\pm 30^\circ$\\ \hline
   \end{tabu}
    \caption{Mission requirements.}
    \label{mission_requirement}
\end{minipage}\hfill
\begin{minipage}[b]{0.62\linewidth}
\centering
\includegraphics[width=80mm]{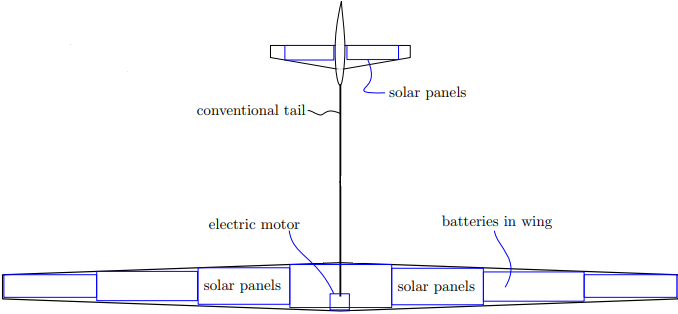}
\captionof{figure}{A simple diagram of the solar aircraft.}
\label{the_aircraft}
\end{minipage}
\end{table}

\begin{center}
\begin{tabular}{ |P{10em}|P{3cm}|P{4cm}|}
\hline
Uncertain Parameter & Value & Description \\
\hline
$\eta_m$ & $0.95 \pm 7\%$ & Motor Efficiency\\
\hline
$h_{batt}$ & $350 \pm 15\%$ $[\frac{Wh}{Kg}]$ & Battery Specific Energy \\
\hline
$\eta_{solar}$ & $0.22 \pm 15\%$ & Solar Cell Efficiency\\
\hline
$V_{wind_{ref}}$ & $100 \pm 3\%$ $[\frac{m}{s}]$ & Reference Wind Speed\\  
\hline
$\eta_{prop}$ & $0.8 \pm 10\%$ & Propeller Efficiency\\
\hline
\end{tabular}
\captionof{table}{A sample of the solar aircraft uncertain parameters.}
\label{uncertain_solar}
\end{center}

This problem was solved for an elliptical uncertainty set for various values of $\Gamma$ with a maximum number of piecewise-linear terms of 50 and a desired relative error of 1\% between the solution of the upper and lower tractable approximations. The average performance was then calculated from 300 realizations.

Table \ref{solar_table} and Figure \ref{compare_solar} show that the number of constraints needed to formulate this approximate RGP with the Two Term formulation is extremely large even at a relative error of 2\%. By comparison, the formulations presented in this paper have a relatively small number of constraints, primarily due to the partitioning of large posynomials into smaller ones.\\

\begin{table}
\begin{center}
\begin{tabular}{ |m{7em}|m{1.8cm}|m{1.5cm}|m{2.2cm}|m{1.8cm}| }
\hline
Method & Uncertainty Set & Relative Error [$\%$] & Number of PWL Sections & Number of Constraints \\
\hline
\multirow{2}{7em}{Two Term} & Box & 2.3 & 50 & 91,599\\ 
& Elliptical & 2.2 & 50 & 91,599\\
\hline
\multirow{2}{7em}{Simple Conservative} & Box & 0 & 0 & 2,896\\ 
& Elliptical & 0 & 0 & 2,896 \\
\hline
\multirow{2}{7em}{Linearized Perturbations} & Box & 0 & 0 & 3,895\\ 
& Elliptical & 0.9 & 31 & 2,899 \\  
\hline
\multirow{2}{7em}{Best Pairs} & Box & 0.7 & 4 & 2,939\\ 
& Elliptical & 0.9 & 32 & 4,433 \\
\hline
Deterministic & N/A & N/A & N/A & 2,837\\
\hline
\end{tabular}
\captionof{table}{Comparing the solar aircraft results using different methodologies.}
\label{solar_table}
\end{center}
\end{table}

\begin{figure}[!ht]
    \centering
    \captionsetup{justification=centering, font=small}
    \begin{subfigure}{0.49\textwidth}
        \centering
        \includegraphics[height=1.81in]{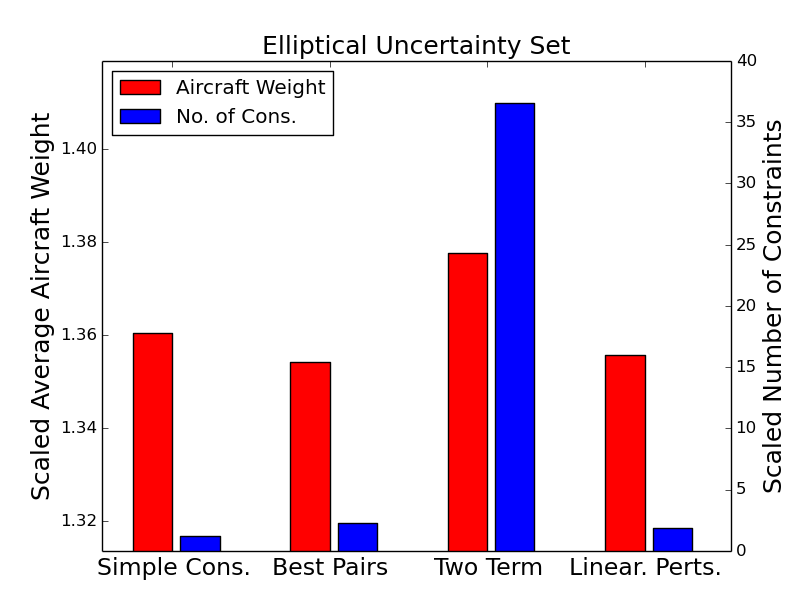}
    \end{subfigure}%
    ~ 
    \begin{subfigure}{0.49\textwidth}
        \centering
        \includegraphics[height=1.81in]{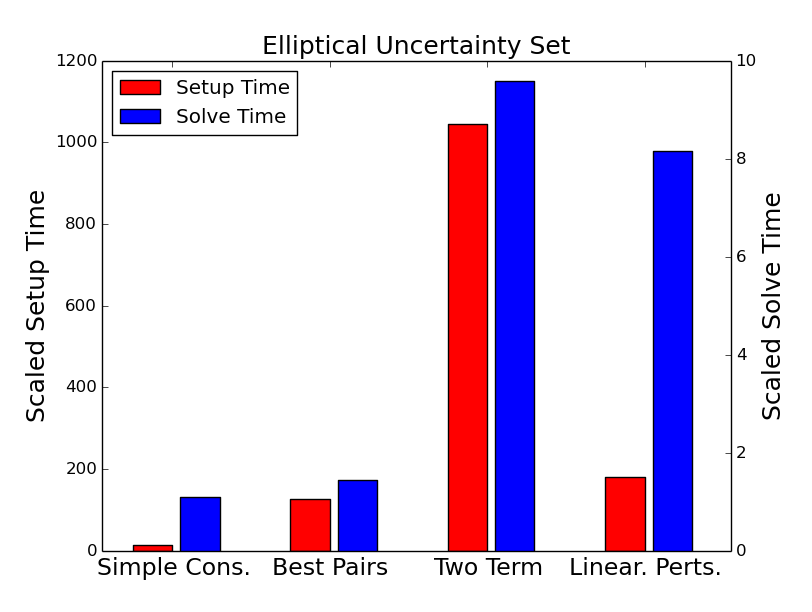}
    \end{subfigure}
    \caption{Robust solar aircraft results scaled by deterministic objective value, number of constraints, and timing.}
    \label{compare_solar}
\end{figure}

\begin{figure}[!htb]
    \centering
    \captionsetup{justification=centering, font=small}
    \includegraphics[scale=0.56]{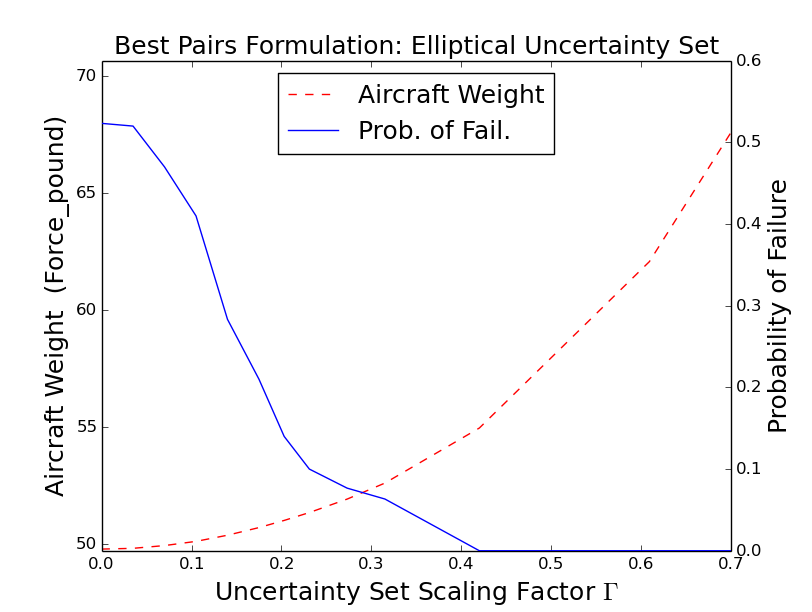}
    \caption{Performance of the solar aircraft design based on the Best Pairs formulation, as a function of $\Gamma$.}
    \label{solar_var_gamma}
\end{figure}

Figure \ref{solar_var_gamma} shows that the probability of failure in the robust design goes to zero as $\Gamma$ increases, while the deterministic design fails with a probability of 0.7.

As shown in Figure \ref{compare_solar}, the programs influenced by our formulations take much less time to  setup and solve than the Two Term formulation (with the exception of the Linearized Perturbations solve), and all obtain better performing designs. Figure \ref{solar_var_pwl} additionally shows the fast convergence of the Best Pairs and Linearized Perturbations formulations with respect to the number of piecewise-linear terms used.

It can be also seen that the Simple Conservative Formulation is significantly the fastest, requires the least number of constraints, and leads to an acceptable performance when compared to the Two-term Formulation.

\begin{figure}[!ht]
    \centering
    \captionsetup{justification=centering, font=small}
    \includegraphics[scale=0.56]{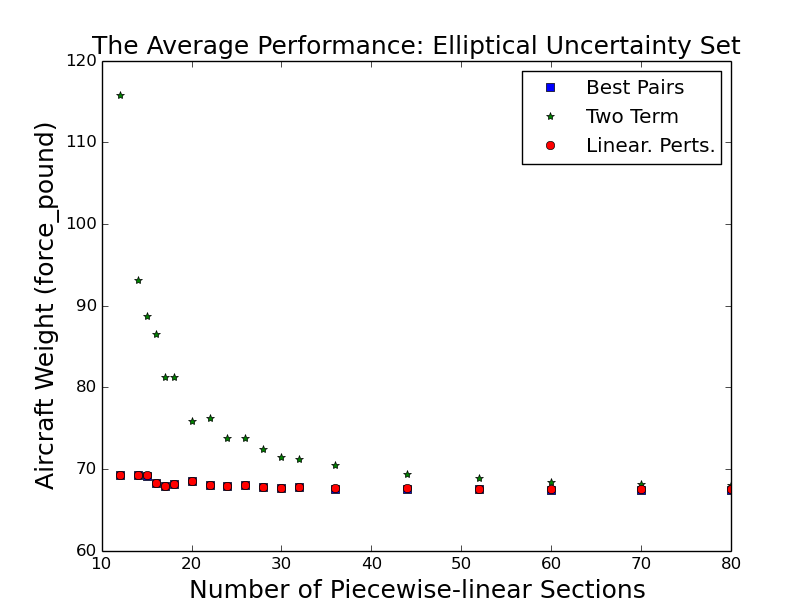}
    \caption{The performance of the solar aircraft design as a function of the number of piecewise-linear terms for elliptical uncertainty set.}
    \label{solar_var_pwl}
\end{figure}
\section{Conclusion} \label{conc}
In this article we have described three different methodologies for creating approximate formulations of robust geometric programs. These methodologies build on each other to derive equivalent but increasingly less conservative formulations which are tractable as GPs, SPs, or conic optimization problems. Several aircraft design problems from the GP literature were robustified and a significant improvement over the state of art \cite{hsiung_kim_boyd_2007} was shown, implying that robust programming may be a promising way to model uncertainties in engineering design. We are currently extending this work to solve Robust Signomial Programming problems.

\appendix

\section{Robust Linear Programming} \label{LP_to_GP}

This section reviews robust linear programming, a building block to formulate a tractable approximate robust geometric program. Two different examples of perturbation sets will be used for clarification: box and elliptical. Those sets will be used for our discussion throughout the article due to the fact that they are frequently used.

Consider the system of linear constraints
\begin{displaymath}
    \mat{A}\vec{x} + \vec{b} \leq 0
\end{displaymath}
where
\begin{equation*}
\begin{aligned}
\mat{A} &\text{ is $m \times n$}\\
\vec{x} &\text{ is $n \times 1$}\\
\vec{b} &\text{ is $m \times 1$}\\
\end{aligned}
\end{equation*}
Assuming that data is uncertain and is given by equations \eqref{Data} and \eqref{perturbation_set}, the constraints should be satisfied for all values of $\vec{\zeta}$, and thus the robust counterparts of the $i^{th}$ linear constraint is:
\begin{equation*}
    \sup_{\vec{\zeta} \in \mathcal{Z}}\left\{ \textstyle{\sum}_{l=1}^{L} \zeta_l([\vec{a}^l_i]^T \vec{x} + b^l_i) \right\} \leq - [\vec{a}^0_i]^T\vec{x} - b^0_i
\end{equation*}
This is equivalent to the optimization problem:
\begin{equation}
    \max_{\vec{\zeta},\vec{u}} \left\{ \textstyle{\sum}_{l=1}^{L} \zeta_l([\vec{a}^l_i]^T \vec{x} + b^l_i): \mat{F}\vec{\zeta} + \mat{G}\vec{u} + \vec{h} \in \mathbf{K} \right\} \leq - [\vec{a}^0_i]^T\vec{x} - b^0_i
\label{linear_counterparts}
\end{equation}
Applying the conic duality theorem, equation \eqref{linear_counterparts} is equivalent to:
\begin{equation}
\begin{aligned}
\vec{h}^T\vec{y}_i + [\vec{a}_i^0]^T\vec{x} + b_i^0 &\leq 0\\
\mat{G}^T\vec{y}_i &= \vec{0}\\
(\mat{F}^T\vec{y}_i)_l + [\vec{a}_i^l]^T\vec{x} + b_i^l &= 0 \quad l = 1,2,...,L\\
\vec{y}_i &\in \mathbf{K}^*
\end{aligned}
\label{robust_linear_general}
\end{equation}
where $\vec{y}_i \in \mathbb{R}^N$, and $\mathbf{K}^*$ is the dual cone of $\mathbf{K}$ (see Appendix \ref{cones}) \cite{ben-tal_ghaoui_nemirovski_2009}.

\subsection{Box Uncertainty Sets} \label{box_linear_app}
If the perturbation set $\mathcal{Z}$ given in equation \eqref{perturbation_set} is a box uncertainty set, i.e. $\|\vec{\zeta}\|_{\infty} \leq \Gamma$, then 
\begin{itemize}
	\item $\mat{F}\vec{\zeta} = [\vec{\zeta};0]$ 
	\item $\mat{G} = \mat{0}$, $\vec{h} = [\vec{0}_{L \times 1};\Gamma]$
	\item $\mathbf{K} = \left\{(\vec{z};t) \in \mathbb{R}^L \times \mathbb{R}: t > \|\vec{z}\|_{\infty} \right\}$ 
	\item the dual cone $\mathbf{K}^* = \left\{(\vec{z};t) \in \mathbb{R}^L \times \mathbb{R}: t > \|\vec{z}\|_{1} \right\}$
\end{itemize}
and therefore, equation \eqref{linear_counterparts} is equivalent to:
\begin{equation}
\Gamma \textstyle{\sum}_{l=1}^L |- {b}^l_{i} - \vec{a}^l_i\vec{x}| + \vec{a}^0_i\vec{x} + b^0_i \leq 0
\label{box_absolute}
\end{equation}
If only $\vec{b}$ is uncertain, i.e. $\mat{A}^l = 0 \quad \forall l = 1,2,...,L$, then equation \eqref{box_absolute} will become:
\begin{equation}
\textstyle{\sum}_{l=1}^L \vec{a}^0_{i}\vec{x} + b^0_{i} + \Gamma \textstyle{\sum}_{l=1}^L |b^l_{i}| \leq 0
\label{box_coeff}
\end{equation}
which is a linear constraint.\\
On the other hand, if $\mat{A}$ is also uncertain, then equation \eqref{box_absolute} is equivalent to the following set of linear constraints:
\begin{equation}
\begin{aligned}
\Gamma \textstyle{\sum}_{l=1}^L w^l_{i} + \vec{a}^0_{i}\vec{x} + b^0_{i} &\leq 0\\
- b^l_{i} - \vec{a}^l_{i}\vec{x} &\leq w^l_{i} &&\forall l \in 1,...,L\\
b^l_{i} + \vec{a}^l_{i}\vec{x} &\leq w^l_{i} &&\forall l \in 1,...,L\\
\end{aligned}
\label{box_linear}
\end{equation}

\subsection{Elliptical Uncertainty Sets} \label{ellip}
If the perturbation set $\mathcal{Z}$ is now an elliptical uncertainty set, i.e. $\textstyle{\sum}_{l=1}^L\frac{\zeta_l}{\sigma_l} \leq \Gamma$, then 
\begin{itemize}
	\item $\mat{F}\vec{\zeta} = [\mat{\sigma}^{-1}\vec{\zeta};\ 0]$ with $\mat{\sigma} = \text{diag}(\sigma_1,...,\sigma_L)$
	\item $\mat{G} = \mat{0}$, $\vec{h} = [\vec{0}_{L \times 1};\ \Gamma]$
	\item $\mathbf{K} = \left\{(\vec{z};t) \in \mathbb{R^L} \times \mathbb{R}: t > \|\vec{z}\|_{2} \right\}$
	\item The dual cone $\mathbf{K}^*$ = $\mathbf{K}$ (Lorentz or second order cone)
\end{itemize}
and therefore, equation \eqref{linear_counterparts} is equivalent to:
\begin{equation}
\Gamma \sqrt{\textstyle{\sum}_{l=1}^L \sigma_l^2(- b^l_{i} - \vec{a}^l_{i}\vec{x})^2} + \vec{a}^0_{i}\vec{x} + b^0_{i} \leq 0
\label{ell_absolute}
\end{equation}
which is a second order conic constraint.\\
If only $\vec{b}$ is uncertain, i.e. $\mat{A}^l = 0 \quad \forall l = 1,2,...,L$, then equation \eqref{ell_absolute} will become:
\begin{equation}
\vec{a}^0_{i}\vec{x} + b^0_{i} + \Gamma \sqrt{\textstyle{\sum}_{l=1}^L \sigma_l^2(b^l_{i})^2} \leq 0
\label{ell_coeff}
\end{equation}
which is a linear constraint ($\sqrt{\textstyle{\sum}_{l=1}^L \sigma_l^2(b^l_{i})^2}$ is a constant).

\section{Cones} \label{cones}
\subsection{Euclidean Space}
An Euclidean space is a finite dimensional linear space over real numbers equipped with an inner product $\langle x,y \rangle_E$.

\subsection{Cones}
A nonempty subset $\textbf{K}$ of an Euclidean space is called a cone if for any $x \in \textbf{K}$ and $\alpha \geq 0$ $\alpha x \in \textbf{K}$ \cite{bertsimas_tsitsiklis_1997}.\\
A cone is said to be convex cone if $\alpha, \beta \geq 0$ and $x, y \in \textbf{K}$, we have $\alpha x + \beta y \in \textbf{K}$ \cite{ben-tal_ghaoui_nemirovski_2009}.

\subsection{Dual Cones}
If $\textbf{K}$ is a cone in an euclidean space $\textbf{E}$, then the set 
\begin{equation*}
    \textbf{K}^* = \left\{ e \in \textbf{E}:\langle e,h \rangle_E \geq 0 \quad \forall h \in \textbf{K} \right\}
\end{equation*}
is also a cone and is called the cone dual to $\textbf{K}$ \cite{ben-tal_ghaoui_nemirovski_2009}.

\section{Equivalence Relations} \label{EqRel}
Let $\mathcal{R}$ be some relation on a set $\mathbf{S}$, and let $x,y \in \mathbf{S}$. We say $x \simeq y$ if $x$ and $y$ are related.
 
A relation $\mathcal{R}$ on a set $\mathbf{S}$ is called an equivalence relation if the following is true
\begin{itemize}
	\item the relation is reflexive, i.e. for all $a \in \mathbf{S}$, $a \simeq a$
	\item the relation is symmetric, i.e. for $a,b \in \mathbf{S}$, if $a \simeq b$, then $b \simeq a$
	\item the relation is transitive, i.e. for $a, b, c \in \mathbf{S}$, if $a \simeq b$ and $b \simeq c$, then $a \simeq c$
\end{itemize}
An equivalence relation naturally partitions a set into equivalence classes. Those classes are such that if $a$ and $b$ belong to the same class then $a \simeq b$, and if $a$ and $b$ are not related, then they belong to different classes.

\section{Signomial Programming} \label{sigProg}
Signomials allow us to solve a non log-convex optimization problem as sequential geometric programs \cite{MARANAS1997351}. A signomial program has the following form:
\begin{equation}
\begin{aligned}
&\text{minimize } &&f_{0}(\mathbf{x})\\
& &&g_{i}(\mathbf{x}) -  h_{i}(\mathbf{x}) &&\leq 0 \quad &i = 1, ...., m
\end{aligned}
\label{sp}
\end{equation}
where $g_{i}$ and $h_{i}$ are posynomials.

To clarify how a signomial program is useful in our discussion, consider equation \eqref{linearCon_linPerts}, and assume that $\mathcal{Z}$ is an elliptical uncertainty set. Using the knowledge from subsection \ref{ellip}, \eqref{linearCon_linPerts} is equivalent to:
\begin{equation}
\begin{aligned}
&\textstyle{\sum}_{k \in S_{i,j}}g_{i,j}^k e^{\vec{a_{ik}}\vec{x} + b^0_{ik}} + e^{0.5s_{i,j}} && \leq e^{t_{ij}} \qquad && \forall (i, j) \in \mathbf{P}\\
&\textstyle{\sum}_{l=1}^L \sigma_l e^{s_{i,j}^{l}} && \leq e^{s_{i,j}} \qquad && \forall (i, j) \in \mathbf{P}\\
&\left(\textstyle{\sum}_{k \in S_{i,j}}f_{i,j}^{k,l}e^{\vec{a_{ik}}\vec{x} + b^0_{ik}}\right)^2 && \leq e^{s_{i,j}^{l}} \qquad \forall l \in 1,...,L \qquad && \forall (i, j) \in \mathbf{P}
\end{aligned}
\label{sp_compatible_constraints}
\end{equation}
If one of the `$f$'s is negative, then some constraints from the third set might not be GP-compatible, but SP-compatible. The robust geometric program will be a signomial program.

\section*{Acknowledgments}
This research was supported by Boeing. We thank Professor Karen Willcox who provided insight, expertise, and important feedback that greatly assisted this work.
\bibliography{main}

\begin{thebibliography}{10}

\bibitem{babakhani_lavaei_doyle_hajimiri_2010}
A~Babakhani, J~Lavaei, J~C Doyle, and A~Hajimiri.
\newblock Finding globally optimum solutions in antenna optimization problems.
\newblock {\em 2010 IEEE Antennas and Propagation Society International
  Symposium}, 2010.

\bibitem{bandi_bertsimas_2012}
Chaithanya Bandi and Dimitris Bertsimas.
\newblock Tractable stochastic analysis in high dimensions via robust
  optimization.
\newblock {\em Mathematical Programming}, 134(1):23–70, 2012.

\bibitem{ben-tal_ghaoui_nemirovski_2009}
A.~Ben-Tal, Laurent~El Ghaoui, and Nemirovskiĭ~Arkadiĭ Semenovich.
\newblock {\em Robust optimization}.
\newblock Princeton University Press, 2009.

\bibitem{bertsimas_sim_2005}
Dimitris Bertsimas and Melvyn Sim.
\newblock Tractable approximations to robust conic optimization problems.
\newblock {\em Springer-Verlag}, Dec 2005.

\bibitem{bertsimas_tsitsiklis_1997}
Dimitris Bertsimas and John~N. Tsitsiklis.
\newblock {\em Introduction to linear optimization}.
\newblock Athena Scientific, 1997.

\bibitem{birge_louveaux_2011}
John~R. Birge and F.~Louveaux.
\newblock {\em Introduction to stochastic programming}.
\newblock Springer Science, 2011.

\bibitem{GP_tutorial}
Stephen Boyd, Seung-Jean Kim, Lieven Vandenberghe, and Arash Hassibi.
\newblock A tutorial on geometric programming.
\newblock {\em Optimization and Engineering}, 8(1):67–127, Oct 2007.

\bibitem{boyd_kim_patil_horowitz_2005}
Stephen~P. Boyd, Seung-Jean Kim, Dinesh~D. Patil, and Mark~A. Horowitz.
\newblock Digital circuit optimization via geometric programming.
\newblock {\em Operations Research}, 53(6):899–932, 2005.

\bibitem{burton_hoburg}
Michael Burton and Warren Hoburg.
\newblock Solar and gas powered long-endurance unmanned aircraft sizing via
  geometric programming.
\newblock {\em Journal of Aircraft}, 55(1):212–225, 2018.

\bibitem{RGPcoNP}
Andre Chassein and Marc Goerigk.
\newblock Robust geometric programming is co-np hard.
\newblock 2014.

\bibitem{chiang_2005}
Mung Chiang.
\newblock {\em Geometric programming for communication systems}.
\newblock Now Publishers, 2005.

\bibitem{hershenson_2004}
M.~del Mar~Hershenson.
\newblock Cmos analog circuit design via geometric programming.
\newblock In {\em Proceedings of the 2004 American Control Conference},
  volume~4, pages 3266--3271 vol.4, June 2004.

\bibitem{bib:Domahidi2013ecos}
A.~Domahidi, E.~Chu, and S.~Boyd.
\newblock {ECOS}: {A}n {SOCP} solver for embedded systems.
\newblock In {\em European Control Conference (ECC)}, pages 3071--3076, 2013.

\bibitem{greenberg_1995}
Harvey~J. Greenberg.
\newblock Mathematical programming models for environmental quality control.
\newblock {\em Operations Research}, 43(4):578–622, 1995.

\bibitem{hoburg_abbeel_2014}
Warren Hoburg and Pieter Abbeel.
\newblock Geometric programming for aircraft design optimization.
\newblock {\em AIAA Journal}, 52(11):2414–2426, 2014.

\bibitem{hoburg_kirschen_abbeel_2016}
Warren Hoburg, Philippe Kirschen, and Pieter Abbeel.
\newblock Data fitting with geometric-programming-compatible softmax functions.
\newblock {\em Optimization and Engineering}, 17(4):897–918, Apr 2016.

\bibitem{woodyThesis}
Warren~Woodrow Hoburg.
\newblock {\em Aircraft Design Optimization as a Geometric Program}.
\newblock PhD thesis, University of California, Berkeley, 2013.

\bibitem{hsiung_kim_boyd_2007}
Kan-Lin Hsiung, Seung-Jean Kim, and Stephen Boyd.
\newblock Tractable approximate robust geometric programming.
\newblock {\em Optimization and Engineering}, 9(2):95–118, Apr 2007.

\bibitem{kortanek_xu_ye_1997}
K.~O. Kortanek, Xiaojie Xu, and Yinyu Ye.
\newblock An infeasible interior-point algorithm for solving primal and dual
  geometric programs.
\newblock {\em Mathematical Programming}, 76(1):155–181, 1997.

\bibitem{MARANAS1997351}
Costas~D. Maranas and Christodoulos~A. Floudas.
\newblock Global optimization in generalized geometric programming.
\newblock {\em Computers \& Chemical Engineering}, 21(4):351 -- 369, 1997.

\bibitem{mazumdar_jefferson_1983}
M.~Mazumdar and T.~R. Jefferson.
\newblock Maximum likelihood estimates for multinomial probabilities via
  geometric programming.
\newblock {\em Biometrika}, 70(1):257, 1983.

\bibitem{nesterov_nemirovskii_1994}
Y.~Nesterov and A.~Nemirovskii.
\newblock {\em Interior-Point Polynomial Algorithms in Convex Programming}.
\newblock Society for Industrial and Applied Mathematics, 1994.

\bibitem{o’donoghue_chu_parikh_boyd_2016}
Brendan O’Donoghue, Eric Chu, Neal Parikh, and Stephen Boyd.
\newblock Conic optimization via operator splitting and homogeneous self-dual
  embedding.
\newblock {\em Journal of Optimization Theory and Applications},
  169(3):1042–1068, 2016.

\bibitem{prékopa_2010}
András Prékopa.
\newblock {\em Stochastic programming}.
\newblock Kluwer Acad. Publ., 2010.

\bibitem{soyster_1973}
A.~L. Soyster.
\newblock Technical note—convex programming with set-inclusive constraints
  and applications to inexact linear programming.
\newblock {\em Operations Research}, 21(5):1154–1157, 1973.

\bibitem{wall_greening_woolsey_1986}
Thomas~Wayne Wall, Doran Greening, and R.~E.~D. Woolsey.
\newblock Or practice—solving complex chemical equilibria using a
  geometric-programming based technique.
\newblock {\em Operations Research}, 34(3):345–355, 1986.

\end{thebibliography}
\bibliographystyle{plain}

\end{document}